\documentclass[a4paper]{amsart}
\usepackage[latin1]{inputenc}
\usepackage{amsfonts}
\usepackage{amsmath,latexsym,amssymb,amsfonts, amsthm}
\usepackage{esint}
\usepackage{cite}
\usepackage{graphicx}
\usepackage{amscd}
\usepackage{color}
\usepackage{bm}             
\usepackage{enumerate}
\usepackage[dvips]{epsfig}
\usepackage{psfrag}
\usepackage{epsfig}

\usepackage[
  hmarginratio={1:1},     
  vmarginratio={1:1},     
  textwidth=15cm,        
  textheight=21cm,
  heightrounded,          
]{geometry}

\usepackage{graphicx,color}
\usepackage[colorlinks]{hyperref}
\hypersetup{linkcolor=blue,citecolor=blue,filecolor=black,urlcolor=blue}

\newtheorem{theorem}{Theorem}

\newtheorem{proposition}[theorem]{Proposition}
\newtheorem{lemma}[theorem]{Lemma}
\theoremstyle{definition}
\newtheorem{remark}[theorem]{Remark}

\newtheorem{definition}[theorem]{Definition}

\numberwithin{theorem}{section}

\numberwithin{equation}{section}

\newcommand{\R}{\mathbb{R}}
\newcommand{\N}{\mathbb{N}}

\newcommand{\pa}{\partial}

\renewcommand{\div}{\,{\rm div}\,}

\newcommand{\cC}{{\mathcal C}}

\newcommand{\cF}{{\mathcal F}}   
   
\newcommand{\cH}{{\mathcal H}}

\newcommand{\cK}{{\mathcal K}}

\newcommand{\cO}{{\mathcal O}}

\newcommand{\cR}{{\mathcal R}}

\newcommand{\cV}{{\mathcal V}}

\newcommand{\dist}{{\rm dist}}
\newcommand{\supp}{{\rm supp}}

\newcommand{\weak}{\rightharpoonup}

\newcommand{\eps}{\varepsilon}

\DeclareMathOperator{\loc}{loc}

\renewcommand{\epsilon}{\varepsilon}

\author[N. Soave]{Nicola Soave}\thanks{}
\address{Nicola Soave \newline \indent
Dipartimento di Matematica,  Politecnico di Milano,  \newline \indent
Via Edoardo Bonardi 9, 20133 Milano, Italy}
\email{nicola.soave@gmail.com; nicola.soave@polimi.it}

\author[S. Terracini]{Susanna Terracini}\thanks{}
\address{Susanna Terracini \newline \indent
 Dipartimento di Matematica ``Giuseppe Peano'', Universit\`a di Torino, \newline \indent
Via Carlo Alberto, 10,
10123 Torino, Italy}
\email{susanna.terracini@unito.it}

\title[The nodal set of solutions to singular equations]{The nodal set of solutions to some elliptic \\ problems: singular nonlinearities}
\keywords{Singular Lane Emden equations, nodal solutions, nodal set, nondegeneracy}
\subjclass[2010]{35B05, 35R35 (35J60, 35B60, 28A78)}
\thanks{The authors are partially supported by the ERC Advanced Grant 2013 n. 339958 ``Complex Patterns for Strongly Interacting Dynamical Systems - COMPAT''. N. Soave is partially supported by the PRIN-2015KB9WPT\texttt{\char`_}010 Grant: ``Variational methods, with applications to problems in mathematical physics and geometry", and by the INDAM-GNAMPA project ``Aspetti non-locali in fenomeni di segregazione".}
\date{\today}

\begin{document}

\maketitle

\begin{abstract}
This paper deals with solutions to the equation 
\begin{equation*}
-\Delta u = \lambda_+ \left(u^+\right)^{q-1} - \lambda_- \left(u^-\right)^{q-1} \quad \text{in $B_1$}
\end{equation*}
where $\lambda_+,\lambda_- > 0$, $q \in (0,1)$, $B_1=B_1(0)$ is the unit ball in $\R^N$, $N \ge 2$, and $u^+:= \max\{u,0\}$, $u^-:= \max\{-u,0\}$ are the positive and the negative part of $u$, respectively. 
We extend to this class of \emph{singular} equations the results recently obtained in \cite{SoTe2018} for  \emph{sublinear and discontinuous} equations, $1\leq q<2$, namely:  (a) the finiteness of the vanishing order at every point and the complete characterization of the order spectrum; (b) a weak non-degeneracy property; (c) regularity of the nodal set of any solution: the nodal set is a locally finite collection of regular codimension one manifolds up to a residual singular set having Hausdorff dimension at most $N-2$ (locally finite when $N=2$). As an intermediate step, we establish the regularity of a class of \emph{not necessarily minimal} solutions.

The proofs are based on a priori bounds, monotonicity formul\ae \ for a $2$-parameter family of Weiss-type functionals, blow-up arguments, and the classification of homogenous solutions. 
\end{abstract}

\section{Introduction}

The purpose of this paper is to study the local regularity of the solutions, and the structure of their nodal sets, for the equation 
\begin{equation}\label{eq}
-\Delta u = \lambda_+ (u^+)^{q-1} - \lambda_- (u^-)^{q-1}
\end{equation}
in the unit ball $B_1$ (or in any other open set), in the  \emph{singular case}, that is when $q \in (0,1)$, $N \ge 2$ and $\lambda_+, \lambda_->0$. We shall deal with nodal solutions: hence, as $q<1$, the right hand side is singular as $u$ approaches $0$. Therefore, solutions can not be intended in a classical, pointwise or weak senses but in a suitable one which will be specified below. 

Singular Lane-Emden equations like \eqref{eq}, and the properties of the nodal sets for their solutions, have been considered in the literature, mainly with a nonlinearity with opposite sign with respect to ours. One and two phases problems of type \eqref{eq} for $q \in (0,1)$, but with $\lambda_\pm\leq 0$, have been studied in \cite{AlPh, LiPe, MatPet, Ph, Ph2} (see also \cite{ArTe} for a fully nonlinear analogue). In this case, nontrivial solutions may vanish on open sets, and no unique continuation principle holds. Nonnegative solutions to singular semilinear elliptic equations of the type $-\Delta u=-u^{q-1}$ with $q <0$ (note the fact that $q$ varies in a different range than ours, and the opposite sign of the nonlinearity), and their vanishing sets, have been studied in connection with applications to electrostatic Micro-Electromechanical System (MEMS) devices or  Van der Waals force driven thin films of viscous fluids (see \cite{HeMaVe,JuLi,DaWaWe,EsGhGu} and references therein). The case $q<0$ differs substantially from $q \in (0,1)$, since the singularity do not appear only in the equation (as for $q \in (0,1)$), but also in the naturally associated functional. The boundary behavior of positive $H^1_0(\Omega)$-solutions to  $-\Delta u=\lambda(x)u^{-\nu}$ for $\nu>0$ was previously established in \cite{GuLi} under suitable smoothness assumptions on $\partial\Omega$ and $\lambda(x)$.

This paper is the natural continuation of our previous one \cite{SoTe2018} (see also \cite{Fo}), where we have faced \eqref{eq} in the sublinear and discontinuous cases, that is when $q \in (1,2)$ and $q=1$ respectively. In that paper we have proved the finiteness of the vanishing order at every zero point and provided the complete characterization of the order spectrum; moreover, we have shown how non-degeneracy holds at each zero point, yielding the  regularity of the nodal set of any solution: indeed we proved that the nodal set is a locally finite collection of regular codimension one manifolds up to a residual singular set having Hausdorff dimension at most $N-2$ (locally finite when $N=2$). Ultimately, the main features of the nodal set are strikingly similar to those of the solutions to linear (or superlinear) equations, with two remarkable differences: first of all, the admissible vanishing orders can not exceed the critical value $2/(2-q)$. At threshold, we found a multiplicity of homogeneous solutions, yielding the non-validity of any estimate of the $(N-1)$-dimensional measure of the nodal set of a solution in terms of the vanishing order.

Compared with the case $1\leq q<2$, a major difficulty of the singular one stays in the possible lack of  the $C^{1,\alpha}$ regularity of solutions (for every $0<\alpha<1$). Indeed, while when $q\in[1,2)$ the regularity of weak solutions is provided directly by elliptic estimates, when $0<q<1$ the regularity is a highly non-trivial issue and, in facts, may depend on the very same definition of solution to \eqref{eq}. Throughout this paper, we will consider the following notions.

\begin{definition}\label{def: good}
Let $\Omega \subset \R^N$ be open, and let $q \in (0,1)$. A function $u \in H^1_{\loc}(\Omega)$ is said to be a \emph{good solution} of \eqref{eq} in $\Omega$ if there exists a sequence $0 < \eps_n \to 0^+$ and corresponding sequences $\{u_n\} \subset H^1_{\loc}(\Omega) \cap C^0(\overline{\Omega})$, $\{f_n\}\subset L^\infty(\Omega)$ such that:
\begin{itemize}
\item[($i$)] for every compact set $K \subset \Omega$, we have $u_n \weak u$ in $H^1(K)$, $f_n\to 0$ almost everywhere,
 and $\{f_n\}$ is uniformly bounded in $L^\infty(\Omega)$;
\item[($ii$)] $u_n$ weakly solves
\[
-\Delta u_n + f_n= \lambda_+ \frac{u_n^+}{\left( \eps_n^2 + u_n^2\right)^\frac{2-q}2}-\lambda_- \frac{u_n^-}{\left( \eps_n^2 + u_n^2\right)^\frac{2-q}2} \quad \text{in $\Omega$}.
\]
\end{itemize}
\end{definition}

Notice that the approximating solutions $u_n$ are of class $C^{1,\alpha}$ for every $\alpha\in(0,1)$, since they solve a regular problem. The natural choice $f_n \equiv 0$ in point ($ii$) of the definition is admissible, but we allow for more generality in order to show that \emph{any local minimizer of the energy functional associated with \eqref{eq} is a good solutions}, see Proposition \ref{prop: min are good} below. We emphasize that all the boundary value problems associated with \eqref{eq} have variational structure, and hence it is very natural to consider class of solutions including minimizers as well as critical points with higher Morse index. In this perspective, we stress that the class of good solutions is in general much wider that the one of local minimizers. As an example, we will consider a homogeneous Dirichlet problem and prove the \emph{existence of infinitely many non-minimal good solutions} via variational methods, see Section \ref{sec: ex many}. 

Our first goal is the regularity of good solutions.

\begin{theorem}[Regularity of good solutions]\label{thm: reg good}
Let $q \in (0,1)$, $\lambda_+,\lambda_->0$, and let $u$ be a good solution to \eqref{eq} in $B_1$. Then $u \in C^{1,\alpha}_{\loc}(B_1)$ for every $0<\alpha<q/(2-q)$, and more precisely for any such $\alpha$
\[
u_n \to u \quad \text{in $C^{1,\alpha}_{\loc}(B_1)$, as $n \to \infty$}.
\] 
Furthermore, $u$ solves \eqref{eq} pointwise in $\{x \in B_1: u(x) \neq 0\}$, and the following domains variation formula holds:
\begin{equation}\label{var dom}
\int_{B_1} \langle dY [\nabla u], \nabla u\rangle- \div Y \left[ \frac12 |\nabla u|^2 - \frac{\lambda_+}q (u^+)^q - \frac{\lambda_-}q (u^-)^q \right] = 0
\end{equation}
for every vector field $Y \in C^\infty_c(B_1)$.
\end{theorem}

\begin{remark}
In facts, we shall prove a stronger results, namely the validity of uniform $C^{1,\alpha}$ bounds, both for approximating and families of good solutions, entailing, as a by product, useful compactness properties, which we will apply to blow-up sequences (Theorem \ref{prop: bounds}). Inspired by the results in \cite{AlPh, Ph, GiGi, LiSi}, where energy minimizing solutions are considered, it is natural to expect that the optimal regularity of good solutions is $C^{1,q/(2-q)}$, which coincides with the natural scaling of the equation. Our theorem, thus, provides a suboptimal result, which could be extended to the optimal one by the arguments in section 6  of \cite{LiSi}.  One of the main feature of our approach is that, as in \cite{SoTe2018}, we shall study the regularity of the nodal set without using the optimal regularity of the solution.
\end{remark}

Besides the application to the study of the nodal set, Theorem \ref{thm: reg good} is interesting per s\'e since it is the first regularity result applying to solutions to \eqref{eq} which are not minimizers, even in a local sense. The optimal regularity of local minimizers follows by the main result in \cite{GiGi} (see also \cite{LiSi}), but the methods there heavily rely on the minimality. The proof of Theorem \ref{thm: reg good} is instead based on a priori bounds on solutions obtained by exploiting a blow-up argument combined with Liouville-type theorems, with a technique inspired by the one introduced in \cite{NoTaTeVe}, and readapted in \cite{DaWaWe}.

Since good solutions of \eqref{eq} are of class $C^{1,\alpha}$, it makes sense to subdivide the nodal set $Z(u):= u^{-1}(\{0\})$ into its \emph{regular part} $\cR(u)$
\[
\cR(u):=\{ x \in B_1: u(x) = 0 \text{ and } \nabla u(x) \neq 0\},
\] 
and its \emph{singular part}
\[
\Sigma(u):= \{ x \in B_1: u(x) = |\nabla u(x)| = 0\};
\]
$\cR(u)$ is in fact locally a $C^{1,\alpha}$ $(N-1)$-dimensional hypersurface by the implicit function theorem, and hence the study of the structure of $Z(u)$ is reduced to that of $\Sigma(u)$. This was the starting point in the analysis carried out in \cite{SoTe2018}. Accordingly, we introduce the notion of vanishing order of a solution in a point, quantifying the rate of growth of $u$ on spheres of varying radii.  To this aim, here and in what follows, for $x_0 \in B_1$ and $0<r<\dist(x_0,\pa B_1)$, we denote $S_r(x_0) = \pa B_r(x_0)$, where $B_r(x_0)$ is the ball of center $x_0$ and radius $r$ (in the frequent case $x_0=0$, we simply write $B_r$ and $S_r$ for the sake of brevity). 

\begin{definition}\label{def: order V}
Let $u$ be a solution to \eqref{eq}, and let $x_0 \in Z(u)$. The \emph{vanishing order of $u$ in $x_0$} is defined as the number $\cV(u,x_0) \in \R^+$ with the property that 
\[
\limsup_{r \to 0^+} \frac1{r^{N-1+2\beta}} \int_{S_r(x_0)} u^2  = \begin{cases} 0 & \text{if $0 <\beta< \cV(u,x_0)$} \\
+\infty  & \text{if $\beta > \cV(u,x_0)$}.
\end{cases}
\]  
If no such number exists, then 
\[
\limsup_{r \to 0^+} \frac1{r^{N-1+2 \beta}} \int_{S_r(x_0)} u^2 = 0 \quad \text{for any $\beta>0$},
\]
and we set $\cV(u,x_0)=+\infty$. 
\end{definition}

As in \cite{SoTe2018} (see Theorems 1.3 and 1.4 therein), we can completely characterize the admissible vanishing orders, and prove at the same time a fundamental non-degeneracy condition.

\begin{theorem}[Classification of the vanishing orders and non-degeneracy]\label{thm: very strong V}
Let $0 < q <1$, $\lambda_+, \lambda_->0$, $u$ be a good solution to \eqref{eq}, and let $x_0 \in Z(u)$. Then 
\[
\cV(u,x_0) \in \left\{1,\frac2{2-q} \right\},
\]
and moreover 
\[
\liminf_{r \to 0^+} \frac{1}{r^{N-1+2 \cV(u,x_0)}} \int_{S_r(x_0)} u^2>0.
\]
\end{theorem}

Note that the finiteness of the vanishing order imply the validity of the \emph{strong unique continuation principle} for the singular Lane-Emden equation, similarly to that recently established in the sublinear and discontinuous cases
\cite{SoWe,Ru,SoTe2018}. The validity of such results strongly depend on the $-$ sign in front of the Laplacian in equation \eqref{eq} (or, equivalently, on the sign of the nonlinearity). Indeed, solutions to
\[
\Delta u = \lambda_+ (u^+)^{q-1} - \lambda_- (u^-)^{q-1}, \qquad \lambda_{\pm}>0, \quad q \in (0,2),
\]
can exhibit dead cores.

Theorem \ref{thm: very strong V}, together with uniform $C^{1,\alpha}$ bounds of blow-up seqences and Weiss type monotonicity formul\ae, allow us to prove existence of homogeneous blow-ups. 

\begin{theorem}[Blow-ups]\label{thm: blow-up}
Let $0 < q <1$, $\lambda_+, \lambda_->0$, $u$ be a good solution to \eqref{eq}, $x_0 \in Z(u)$, and $R \in (0, \dist(x_0,\pa B_1))$. Then the following alternative holds:
\begin{itemize}
\item[($i$)] if $\cV(u,x_0)=1$ then $\nabla u(x_0)\neq 0$, and there exists a function $\Gamma_{x_0}$, such that
\[
u(x) =\langle \nabla u(x_0),x-x_0\rangle + \Gamma_{x_0}(x) \qquad \text{in $B_R(x_0)$},
\]
with 
\[
\begin{cases} \Gamma_{x_0}(x)| \le C |x-x_0|^{1+\alpha}  \\ 
|\nabla \Gamma_{x_0}(x)| \le C |x-x_0|^{\alpha} \end{cases}
 \qquad \text{in $B_R(x_0)$}
\]
for a suitable constant $C$ depending on $\alpha\in (0,q/(2-q))$;
\item[($ii$)] if $\cV(u,x_0) = 2/(2-q)$ then $\nabla u(x_0)= 0$, and, for every sequence $0<r_n \to 0^+$, we have, up to a subsequence,
\[
\frac{u(x_0+r_n x)}{\left( \frac{1}{r^{N-1}} \int_{S_r(x_0)} u^2\right)^\frac12}
 \to \bar u \qquad \text{in $C^{1,\alpha}_{\loc}(\R^N)$ for every }0<\alpha<\frac{q}{2-q},
\]
where $\bar u$ is non-trivial and $2/(2-q)$-homogeneous, solves point-wisely
\[
-\Delta \bar u= \mu \left( \lambda_+ (\bar u^+)^{q-1} -\lambda_- (\bar u^-)^{q-1} \right) \quad \text{in $\R^N \setminus \{\bar u=0\}$}
\]
for some $\mu > 0$, and satisfies the variation of domain formula
\[
\int_{\R^N} \langle dY [\nabla \bar u], \nabla \bar u\rangle- \div Y \left[ \frac12 |\nabla \bar u|^2 - \frac{\mu \lambda_+}q ({\bar{u}}^+)^q - \frac{\mu \lambda_-}q ({\bar{u}}^-)^q \right] = 0
\]
for every $Y \in C^\infty_c(\R^N,\R^N)$.
\end{itemize}
\end{theorem}

As in \cite{SoTe2018}, alternative $(ii)$ of Theorem \ref{thm: blow-up} allows us to estimate the Hausdorff dimension of the singular set, via Federer's dimension reduction principle, after a classification of such homogenous profiles.

\begin{theorem}[Hausdorff dimension of nodal and singular set]\label{thm: Hausdorff}
Let $0 < q <1$, $\lambda_+, \lambda_->0$, $u$ be a good solution to \eqref{eq}, and $x_0 \in Z(u)$. The nodal set $Z(u)$ has Hausdorff dimension $N-1$, and is composed into a regular part $\cR(u)$, locally a $C^{1,\alpha}$ $(N-1)$-dimensional hypersurface, and a singular subset $\Sigma(u)$ with Hausdorff dimension at most $N-2$. Furthermore, if $N=2$ the singular set $\Sigma(u)$ is discrete.
\end{theorem}

Focusing on the range $q \in (0,1)$, we point out that up to our knowledge this is the first result regarding the structure of the nodal set of \emph{sign-changing} and \emph{not necessarily minimal} solutions to singular equations, \emph{holding in arbitrary dimension}. Previous results concern nonnegative minimizers \cite{AlPh, Ph2}, or sign-changing minimizers in dimension $2$ \cite{LiPe}. More in details, Phillips \cite{Ph2} and Alt and Phillips \cite{AlPh} considered minimizers of the one phase problem
\begin{equation}\label{AlPh}
\Delta u = u^{q-1}, \qquad u \ge 0,
\end{equation}
and showed that the free boundary $\Gamma^+:= \pa \{u>0\}$ is $C^\infty$, up to a singular set $\Sigma$ of zero $(N-1)$-dimensional Hausdorff measure. Moreover, they showed that in dimension $N=2$ the singular set $\Sigma$ is empty. Lindgren and Petrosyan proved that for minimizers of the two phases problem
\begin{equation}\label{LiPe}
\Delta u =\lambda_+ (u^+)^{q-1} - \lambda_-(u^-)^{q-1}, \qquad \lambda_{\pm}>0
\end{equation}
(this is \eqref{eq} with the opposite sign in front of the Laplacian) in dimension $2$, the free boundaries $\pa \{u>0\}=\Gamma^+$ and $\pa \{u<0\}=\Gamma^-$ are both $C^1$ regular, and described their behavior close to branch points. We point out that solutions to both \eqref{AlPh} and \eqref{LiPe} in general have dead core (that is, $\{u=0\}$ has non-empty interior), due to the different sign of the nonlinearity. Thus, comparing these results with Theorem \ref{thm: Hausdorff}, it emerges a deep difference between the \emph{stable case} ($+$ in front of $\Delta$) and the \emph{unstable} one ($-$ in front of $\Delta$). This was already pointed out for the range $q \in [1,2)$, compare for instance the results in \cite{ShUrWe, ShWe} and in \cite{SoTe2018}.

\subsection{Structure of the paper}

The first part of the paper will be devoted to the proof of a priori bounds for solutions to equations of type, which will imply Theorem \ref{thm: reg good} as a corollary. Once that Theorem \ref{thm: reg good} is proved, so that we have $C^{1,\alpha}$ regularity and the variation of domain formula \eqref{var dom}, Theorems \ref{thm: very strong V}-\ref{thm: Hausdorff} follow suitably modifying the argument already presented in details in \cite{SoTe2018} for the sublinear case $q \in [1,2)$. We shall only describe the necessary modifications, and this will be the content of Section \ref{sec: changes}. The key point is that from many point of view formula \eqref{var dom} replace equation \eqref{eq} in a singular framework. This idea was already used in the literature e.g. in \cite{DaWaWe, LiPe, TavTer}, and motivate the following definition, partially borrowed from \cite{DaWaWe}:
\begin{definition}
We write that $u \in H^1_{\loc}(\Omega) \cap C^0(\Omega)$ is a \emph{stationary solution} to \eqref{eq} in a domain $\Omega$ if $u$ solves \eqref{eq} point-wisely in $\Omega \setminus \{u=0\}$, and moreover the domain variation formula \eqref{var dom} holds.
\end{definition}
We also introduce the notations
\[
C^{1,{\frac{q}{2-q}}-}(\Omega):= \bigcap_{0<\alpha<\frac{q}{2-q}} C^{1,\alpha}(\Omega),
\]
and
\begin{equation}\label{def norm}
\|u\|_{x_0,r}:= \left( \frac1{r^{N-2}} \int_{B_r(x_0)} |\nabla u|^2\, dx + \frac{1}{r^{N-1}} \int_{S_r(x_0)} u^2 \, d\sigma\right)^\frac12.
\end{equation}
For any $0<r<\dist(x_0,\pa B_1)$ fixed, this is a norm in $H^1(B_r(x_0))$, equivalent to the standard one by trace theory and Poincar\'e's inequality.

\medskip

With this new terminology, Theorem \ref{thm: reg good} establishes that any good solution to \eqref{eq} is of class $C^{1,{\frac{q}{2-q}}-}$, and is a stationary solution to \eqref{eq}. 

As already mentioned, we obtain Theorem \ref{thm: reg good} as a corollary of a more general result:

\begin{theorem}[Compactness of regularizing and blow-up sequences]\label{prop: bounds}
Let $q \in (0,1)$ and $\lambda_+,\lambda_->0$, and $0<\alpha<q/(2-q)$. Suppose that one of the following alternative occurs:
\begin{itemize}
\item[(h1)] $u$ is a good solution to \eqref{eq} in $B_\rho$ for some $\rho>0$, and $\{\eps_n\}$ and $\{u_n\}$ denote the two sequences in Definition \ref{def: good}.
\item[(h2)] $u$ is a stationary solution to \eqref{eq} of class $C^{1,\alpha}(B_1)$; also, the sequence $\{u_n\}$, defined by 
\begin{equation}\label{def blow up} 
u_n(x) := \frac{u(x_0+{\rho}_n x)}{\|u\|_{x_0,\rho_n}} 
\end{equation}
for some $x_0 \in B_1$ and $\rho_n \to 0^+$, is bounded in $H^1(B_\rho)$ for some $\rho>0$; and moreover the sequence
\[
\alpha_n:= \left( \frac{\rho_n^\frac{2}{2-q}}{\|u\|_{x_0,\rho_n}}\right)^{2-q}
\]
is bounded.
\end{itemize}
Then the sequence $\{u_n\}$ is uniformly bounded in $C^{1,\alpha}_{\loc}(B_\rho)$.
\end{theorem}

\begin{remark}\label{rmk: on bounds (0,1)}
The role of the Theorem in twofold: let $u$ be a good solution, so that (h1) is in force; then Theorem \ref{prop: bounds} directly implies the $C^{1,\alpha}_{\loc}$ convergence of the approximating sequence $\{u_n\}$ to $u$, up to a subsequence. Since each $u_n$ is regular, and satisfies a domain variation formula similar to \eqref{var dom}, this yields Theorem \ref{thm: reg good}.

Suppose now that Theorem \ref{prop: bounds} is proved when assumption (h1) is in force, and let $u$ be a good solution. Then, as observed, $u$ is a stationary solution of class $C^{1,\alpha}$ for every $0<\alpha<q/(2-q)$. Defining the blow-up sequence $\{u_n\}$ as in \eqref{def blow up}, if $\{\alpha_n\}$ is bounded, Theorem \ref{prop: bounds} establishes the $C^{1,\alpha}_{\loc}$ convergence of the blow-ups, a fundamental result which we invoke in the proof of Theorem \ref{thm: Hausdorff}. 
\end{remark}

The proof of Theorem \ref{prop: bounds} is divided into two main steps. Using the notation of Definition \ref{def: good}, at first we show in Section \ref{sec: H bounds} that the approximating sequence $\{u_n\}$ not only converges weakly in $H^1_{\loc}$, but is also uniformly bounded in $C^{0,\alpha}_{\loc}$ for every $0<\alpha<1$. Using these bounds, in a second step we show that $\{u_n\}$ is in fact bounded in $C^{1,\alpha}_{\loc}$ provided that $0<\alpha<q/(2-q)$; this entails in one shot the regularity of $u$ and the validity of the variation of domain formula (by approximation).

As announced, we also show that any local minimizer for the functional associated with \eqref{eq} is a good solution, see Proposition \ref{prop: min are good}, which will be the content of Section \ref{sec: min are good}.

Finally, in Section \ref{sec: ex many} we prove the existence of infinitely many non-minimal good solutions for homogeneous Dirichlet problems.

\section{H\"older bounds for approximating sequence}\label{sec: H bounds}

In this section we use the notation
\[
g_{n}(s):= \lambda_+ \frac{s^+}{\left( \eps_n^2 + s^2\right)^\frac{2-q}2}-\lambda_- \frac{s^-}{\left( \eps_n^2 + s^2\right)^\frac{2-q}2}, 
\]
and prove the validity of the following statement:

\begin{proposition}\label{prop: Hol 1}
Under the assumptions of Theorem \ref{prop: bounds}, we have that $\{u_n\}$ is bounded in $C^{0,\alpha}_{\loc}(B_\rho)$, for every $\alpha \in (0,1)$. 
\end{proposition}

The proof will take the rest of the section. We start with a preliminary lemma.

\begin{lemma}\label{lem: unif l infty}
Under the assumptions of Theorem \ref{prop: bounds}, we have that $\{u_n\}$ is bounded in $L^\infty_{\loc}(B_\rho)$.
\end{lemma}

\begin{proof}
We suppose at first that (h1) holds. Let us consider $w_n:=(u_n-1)^+ \in H^1(B_\rho)$. By Kato's inequality, we have that
\begin{align*}
-\Delta w_n &\le \left( - f_n + \lambda_+ \frac{u_n^+}{\left( \eps_n^2 + u_n^2\right)^\frac{2-q}2}-\lambda_- \frac{u_n^-}{\left( \eps_n^2 + u_n^2\right)^\frac{2-q}2}\right) \chi_{\{u_n >1\}} \\
& \le  \Vert f_n\Vert_{L^\infty(B_\rho)}+ \lambda_+ \le M
\end{align*}
in $B_\rho$. Let $\psi_n$ be the solution to 
 \[
 \begin{cases}
 -\Delta \psi_n = M & \text{in $B_\rho$} \\
 \psi_n = w_n & \text{on $\pa B_\rho$};
 \end{cases}
 \]
applying the comparison principle, we obtain the estimate $\psi_n \ge w_n$ in $B_\rho$. Using the boundedness of $\{w_n\}$ in $L^2(S_\rho)$ (which holds by compactness of traces and $H^1$ convergence), the local boundedness of $\{w_n\}$ in $L^\infty$ follows by the explicit representation of $\psi_n$ via the Newtonian potential, formula (4.10) in \cite{GT}. The same argument applies to $(u_n+1)^-$, and gives the thesis under assumption (h1).

If (h2) holds, nothing changes, since each $u_n$ is continuous in $\overline{B_{\rho}}$, $\{u_n\}$ is bounded in $L^2(S_\rho)$, and
\[
-\Delta u_n = \alpha_n \left( \lambda_+ (u_n^+)^{q-1} -\lambda_- (u_n^-)^{q-1} \right) \quad \text{in the open set $\{u_n \neq 0\}$}.
\]
Hence $-\Delta (u_n \mp 1)^{\pm} \le (\sup_n \alpha_n) \lambda_{\pm}$ in $B_\rho$.
\end{proof}

Let $\alpha \in (0,1)$, let us fix a compact set $K \subset B_\rho$, and let $\eta \in C^\infty_c(B_\rho)$ with $\eta \equiv 1$ on $K$, $0 \le \eta \le 1$. If we show that the sequence $\{(\eta u_n)\}$ is bounded in $C^{0,\alpha}(\overline{B_\rho})$, then, using the fact that $\eta \equiv 1$ on $K$, we infer that $\{u_n\}$ is bounded in $C^{0,\alpha}(K)$, which is the desired result. Thus, from now on we focus on the boundedness of $\{(\eta u_n)\}$. Suppose by contradiction that $\{(\eta u_n)\}$ is not bounded in $C^{0,\alpha}(\overline{B_\rho})$. Since $\eta u_n$ is sufficiently regular for each $n$ fixed, we have that in this case there exist $x_n, y_n \in B_\rho$ such that
\[
L_n:= \sup_{\substack{x \neq y, \\ x, y \in \overline{B_\rho}}} \frac{|(\eta u_n)(x)-(\eta u_n)(y)|}{|x-y|^\alpha} = \frac{|(\eta u_n)(x_n)-(\eta u_n)(y_n)|}{|x_n-y_n|^\alpha} \to +\infty.
\]
Notice that $r_n:= |x_n-y_n| \to 0$ as $n \to \infty$, since $\{(\eta u_n)\}$ is bounded in $L^\infty(B_\rho)$. 

In order to reach a contradiction, we introduce two blow-up sequences:
\[
\tilde v_n(x):= \frac{1}{L_n r_n^\alpha} (\eta u_n)(x_n+r_n x), \qquad \tilde w_n(x):= \frac{\eta(x_n)}{L_n r_n^\alpha} u_n(x_n+r_n x), 
\]
both defined on the scaled domains $(B_\rho-x_n)/r_n =: \Omega_n$. Notice that, since $x_n \in \supp \, \eta \subset B_\rho$, as $n \to \infty$ there holds that $\Omega_n$ exhausts $\R^N$\footnote{It can be convenient to observe that not only $x_n \in \supp\,\eta$, but we can suppose $\eta(x_n)>0$ for every $n$ (otherwise we exchange the role of $x_n$ and $y_n$).}. In what follows we collect some basic properties of $\tilde v_n$ and $\tilde w_n$. We anticipate that, by definition, $\{\tilde v_n\}$ will have good compactness properties. On the other hand, each $\tilde v_n$ does not satisfy a nice equation, and this is why we also consider $\tilde w_n$. The comparison between $\tilde v_n$ and $\tilde w_n$ will ultimately lead to a contradiction. By definition, we have that
\[
[\tilde v_n]_{C^{0,\alpha}(\overline{\Omega_n})} = \sup_{\substack{x \neq y, \\ x, y \in \overline{\Omega_n}}} \frac{|\tilde v_n(x)-\tilde v_n(y)|}{|x-y|^\alpha} = \frac{\left|\tilde v_n(0)-\tilde v_n\left(\frac{y_n-x_n}{r_n}\right)\right|}{\left|\frac{y_n-x_n}{r_n}\right|^\alpha} = 1,
\]
that is, $\tilde v_n$ is globally H\"older continuous in $\overline{\Omega_n}$, with seminorm exactly equal to $1$. In case (h1) holds, letting 
\[
\tilde f_n(x):= -\frac{r_n^{2-\alpha} \eta(x_n)}{L_n} f_n(x_n+r_nx),
\]
\[
k_n:=r_n^{2-\alpha(2-q)} \left( \frac{\eta(x_n)}{L_n}\right)^{2-q}, \qquad \tilde \eps_n := \frac{\eta(x_n) \eps_n}{L_n r_n^{\alpha} },
\]
and
\[
\tilde g_n(\tilde w_n):= \lambda_+ \frac{\tilde w_n^+}{\left({\tilde \eps_n}^2 + \tilde w_n^2 \right){\frac{2-q}2}} - \lambda_- \frac{\tilde w_n^-}{\left( {\tilde \eps_n}^2 + \tilde w_n^2 \right)^{\frac{2-q}2}},
\]
the equation of $\tilde w_n$ can be written in the form
\begin{subequations}\label{eq w tilde}
\begin{equation}\label{eq w tilde 1}
-\Delta \tilde w_n = \tilde f_n + k_n \tilde g_n(\tilde w_n).
\end{equation}

In case (h2) holds, the equation of $u$ is satisfied only in $\{u \neq 0\}$, and hence 
\begin{equation}\label{eq w tilde 2}
-\Delta \tilde w_n = k_n \tilde g_n(\tilde w_n) \qquad \text{in $\Omega_n \setminus \{\tilde w_n = 0\}$},
\end{equation}
\end{subequations}
with $k_n$ as above and
\[
\tilde g_n(\tilde w_n):= \alpha_n \left( \lambda_+ (\tilde w_n^+)^{q-1} - \lambda_- (\tilde w_n^-)^{q-1} \right).
\]

It is convenient to observe that, since $r_n \to 0^+$, $L_n \to \infty$, $\alpha \in (0,1)$, and $\{f_n\}_n$ is uniformly bounded, there holds 
\[
\|\tilde f_n\|_{L^\infty(\Omega_n)} \to 0, \quad \text{and} \quad k_n \to 0^+ \quad \text{as $n \to \infty$}.
\]

We aim at passing to the limit in the equations \eqref{eq w tilde}, but the problem is that, contrarily to $\tilde v_n$, we do not have a priori a precise control on the growth of $\tilde w_n$. This obstruction can be overcome with the following observation.

\begin{lemma}\label{lem: osc}
It results that $\tilde v_n-\tilde w_n \to 0$ as $n \to \infty$, uniformly on compact sets of $\R^N$. Moreover, $\{\tilde w_n\}$ is equicontinuous and has uniformly bounded oscillation on any compact set $K \subset \R^N$.
\end{lemma}

\begin{proof}
Let $K \subset \R^N$ be compact, and let $n$ so large so that $K \subset \Omega_n$. Then, by Lemma \ref{lem: unif l infty} and the Lipschitz continuity of $\eta$, we have  
\[
\sup_{x \in K} | \tilde v_n-\tilde w_n| = \sup_{x \in K} \frac{|u_n(x_n+r_n x)|}{L_n r_n^\alpha} |\eta(x_n+r_n x) - \eta(x_n)| \le \frac{C r_n^{1-\alpha}}{L_n} \sup_{x \in K}|x| \to 0
\]
as $n \to \infty$, as desired.

Regarding the equicontinuity of $\{\tilde w_n\}$, it follows easily by the same property for $\{\tilde v_n\}$ and by local uniform convergence. Indeed, let $K \subset \R^N$ be compact and $x,y \in K$. Then 
\begin{align*}
 |\tilde w_n(x)-\tilde w_n(y)| &\le  |\tilde w_n(x)-\tilde v_n(x)| + |\tilde v_n(x)-\tilde v_n(y)| + |\tilde v_n(y)-\tilde w_n(y)| \\
& \le o(1) +  |x-y|^\alpha,
\end{align*}
with $o(1) \to 0$ uniformly on $K$.
\end{proof}
As a consequence:
\begin{lemma}
The sequence $\{\tilde w_n(0)\}$ is bounded in $\R$.
\end{lemma}
\begin{proof}
Suppose by contradiction that $\{\tilde w_n(0)\}$ is unbounded. Without loss of generality, we can assume that $\tilde w_n(0) \to +\infty$. By Lemma \ref{lem: osc}, we have that $\tilde w_n \to +\infty$ uniformly on compact sets. In particular, this implies that if $K$ is compact, then $\{w_n =0\} \cap K = \emptyset$ for $n$ large. Let now 
\[
\tilde V_n(x) := \tilde v_n(x)-\tilde v_n(0), \quad \tilde W_n(x):= \tilde w_n(x)-\tilde w_n(0).
\]
These sequences are bounded in $x=0$, and hence by definition of $\tilde v_n$ and Lemma \ref{lem: osc} are locally uniformly convergent, up to a subsequence, to the same limit $V$. Notice that $\Delta \tilde W_n= \Delta \tilde w_n$, and since $\tilde w_n \to +\infty$ locally uniformly we can pass to the limit in equation \eqref{eq w tilde}, deducing that the limit $V$ is harmonic in $\R^N$. By convergence, $V$ is globally $\alpha$-H\"older continuous (with H\"older constant equal to $1$), and this implies that necessarily $V$ is constant by the Liouville theorem. On the other hand, up to a subsequence $(y_n-x_n)/r_n \to z \in S_1$, and using once again the uniform convergence we deduce that
\[
V(z)- V(0) = \lim_{n \to \infty} \left( \tilde V_n\left(\frac{y_n-x_n}{r_n}\right) - \tilde V_n(0) \right) = 1;
\]
therefore, $V$ cannot be constant, and we reached the desired contradiction.
\end{proof}

The previous lemma implies that  $\tilde v_n$ and $\tilde w_n$ are bounded in $0$, and hence, using also Lemma \ref{lem: osc}, they are both uniformly convergent on any compact set of $\R^N$ to the same limit function $\tilde v$. Directly by convergence, the function $\tilde v$ is globally $\alpha$-H\"older continuous, with H\"older constant equal to $1$, and is non-constant.

\begin{lemma}
In case (h1) holds, we have that $\tilde \eps_n \to 0$ as $n \to \infty$.
\end{lemma}

\begin{proof}
If up to a subsequence $\tilde \eps_n \ge \bar \eps>0$, using \eqref{eq w tilde 1} we obtain
\[
|\Delta \tilde w_n| \le  |\tilde f_n|+ \frac{k_n}{\bar \eps^{2-q}} |\tilde w_n| \quad \text{in $\Omega_n$}.
\]
Recalling that $k_n \to 0$, $\|\tilde f_n\|_{L^\infty(\Omega_n)} \to 0$, and using the local boundedness of $\tilde w_n$, we deduce that $\tilde v$ is harmonic in $\R^N$. The Liouville theorem implies then that $\tilde v$ must be constant, a contradiction.
\end{proof}

\begin{lemma}\label{lem: ar non 0}
The limit function $\tilde v$ is harmonic in the open set $\{\tilde v \neq 0\}$. 
\end{lemma}

\begin{proof}
The fact that $\{\tilde v \neq 0\}$ is open follows from the continuity of $\tilde v$. Let now $\omega$ be compactly contained in $\{\tilde v \neq 0\}$. For concreteness, let us suppose that $\tilde v>\delta$ in $\overline{\omega}$. Then, by uniform convergence, $\tilde w_n \ge \delta/2$ in $\overline{\omega}$, and therefore equation \eqref{eq w tilde} pass to the limit, with the right hand side which tends to $0$ uniformly in $\overline{\omega}$. That is, $\tilde v$ is harmonic in $\omega$. 
\end{proof}

The behavior of $\tilde v$ on $\{\tilde v=0\}$ will be controlled using a suitable variation of domain formula. To this purpose, we shall prove that $\tilde w_n$ converges to $\tilde v$ strongly in $H^1_{\loc}(\R^N)$. A first step in this direction is the following:

\begin{lemma}\label{lem: weak conv Hol1}
Up to a subsequence, we have that $\tilde w_n \weak \tilde v$ weakly in $H^1_{\loc}(\R^N)$.
\end{lemma}

\begin{proof}
We show that $\{\tilde w_n\}$ is bounded in $H^1(K)$, where $K$ is an arbitrary compact set of $\R^N$. We test equation \eqref{eq w tilde} with $\tilde w_n \varphi^2$, where $\varphi \in C^\infty_c(\R^N)$ is a cut-off function with $\varphi \equiv 1$ in $K$. In case (h1) holds, we have by \eqref{eq w tilde 1}
\begin{align*}
\int_{\Omega_n} |\nabla \tilde w_n|^2 \varphi^2 + 2 \tilde w_n \varphi\langle \nabla \tilde w_n,  \nabla \varphi \rangle&= \int_{\Omega_n} \tilde f_n \tilde w_n \varphi^2 + k_n \tilde g_n(\tilde w_n) w_n \varphi^2 \\
& \le C \|\tilde f_n\|_{L^\infty(\Omega_n)} + C k_n \int_{\Omega_n} \frac{\tilde w_n^2 \varphi^2 }{\left( \tilde \eps_n^2 + \tilde w_n^2 \right)^\frac{2-q}2} \\
& \le C \|\tilde f_n\|_{L^\infty(\Omega_n)} + C k_n \int_{\supp \, \varphi} |\tilde w_n|^{q} \to 0 
\end{align*}
as $n \to \infty$, where we used the facts that $k_n\to 0$, and $\{\tilde w_n\}$ is locally bounded. Therefore, 
\[
\int_{K} |\nabla \tilde w_n|^2 \le \int_{\Omega_n} |\nabla (\tilde w_n \varphi)|^2 \le o(1) + \int_{\supp \, \varphi} \tilde w_n^2 |\nabla \varphi|^2 \le C,
\]
as desired. If (h2) holds, we have to face the additional difficulty that the equation for $\tilde w_n$ is satisfied only in $\{ \tilde w_n \neq 0\}$. We observe at first that
\[
-\Delta |\tilde w_n| = \alpha_n \left( \lambda_+ (\tilde w_n^+)^{q-1} + \lambda_- (\tilde w_n^-)^{q-1} \right) \qquad \text{in $\{|\tilde w_n|>0\}$}.
\]
Then we take a family of smooth functions $\psi_\delta: [0,+\infty) \to [0,1]$, with $\delta>0$ small, such that
\[
\psi_\delta(t) = 0 \quad \text{if $0 \le t \le \delta$}, \quad \psi_\delta(t) = 1 \quad \text{if $t \ge 2\delta$}, \quad \psi_\delta'(t) \ge 0 \quad \text{for every $t$}.
\] 
If $\varphi \in C^\infty_c(\R^N)$ with $\varphi \equiv 1$ on $K$, then $|\tilde w_n| \varphi^2 \psi_\delta(|\tilde w_n|)  \in H_0^1(\{|\tilde w_n|>0\})$ is an admissible nonnegative test function for the equation of $|\tilde w_n|$, and hence we obtain
\begin{multline*}
\int_{\{|\tilde w_n| >0\}} \psi_\delta(|\tilde w_n|) \varphi^2 |\nabla |\tilde w_n||^2 + 2 |\tilde w_n|\psi_\delta(|\tilde w_n|) \varphi \langle \nabla |\tilde w_n|, \varphi \rangle + |\tilde w_n| \varphi^2 \psi_\delta'(|\tilde w_n|) |\nabla |\tilde w_n||^2 \\
\le C k_n \int_{\{|\tilde w_n|>0\}} \psi_\delta(|\tilde w_n|) |\tilde w_n|^{q} \varphi^2.
\end{multline*}
Since the last term on the left hand side is nonnegative, we deduce that
\[
\int_{\{|\tilde w_n| >0\}} \psi_\delta(|\tilde w_n|) \varphi^2 |\nabla |\tilde w_n||^2 + 2|\tilde w_n|\psi_\delta(|\tilde w_n|) \varphi \langle \nabla |\tilde w_n|, \varphi \rangle
\le C k_n \int_{\{|\tilde w_n| >0\}} \psi_\delta(|\tilde w_n|) |\tilde w_n|^{q} \varphi^2,
\]
and, passing to the limit as $\delta \to 0^+$, we obtain
\[
\int_{\Omega_n} \varphi^2 |\nabla |\tilde w_n||^2  + 2 |\tilde w_n| \varphi\langle \nabla |\tilde w_n|,  \nabla \varphi \rangle \le C k_n \int_{\supp \, \varphi}  |\tilde w_n|^{q}.
\]
Therefore, also assuming (h2) we have that
\[
\int_{K} |\nabla \tilde w_n|^2 = \int_{K} |\nabla |\tilde w_n||^2  \le \int_{\Omega_n} |\nabla (|\tilde w_n| \varphi)|^2 \le o(1) + \int_{\supp \, \varphi} |\tilde w_n|^2 |\nabla \varphi|^2 \le C,
\]
and the thesis follows.
\end{proof}

Let us now introduce, for $w \in H^1_{\loc}(\R^N)$, 
\[
H(w,x_0,r) := \int_{S_r(x_0)} w^2.
\]

\begin{lemma}\label{lem: der H 2}
Let $w \in H^1_{\loc}(\R^N) \cap C^0(\R^N)$ such that
\[
-\Delta w = 0 \quad \text{in $\R^N \setminus \{w=0\}$}.
\]
Then the function $H(w,x_0,\,\cdot)$ is absolutely continuous, and 
\[
\frac{d}{dr} \left( \frac{H(w,x_0,r)}{r^{N-1}} \right) =  \frac{2}{r^{N-1}} \int_{B_r(x_0)} |\nabla w|^2.
\]
for every $x_0 \in \R^N$ and $r>0$.
\end{lemma}

\begin{proof}
By regularity, $w \in C^2(\R^N \setminus \{w=0\})$. Thus, if we define $w_\eps^+:= (w-\eps)^+$, it is not difficult to check that $(w_\eps^+)^2$ is of class $C^1(\R^N)$ for every $\eps>0$. Notice also that
\[
\nabla w_\eps^+ = \begin{cases} \nabla w & \text{in $\{w>\eps\}$} \\ 0 & \text{in $\{w \le \eps\}$}, \end{cases}  
\quad \Delta w_\eps^+ = \begin{cases} \Delta w & \text{in $\{w>\eps\}$} \\ 0 & \text{a.e. in $\{w \le \eps\}$}, \end{cases}.
\]
Hence the derivative of $H(w_\eps,x_0,\,\cdot)$ can be computed directly, and we have
\begin{align*}
\frac{d}{dr} \left( \frac{H(w_\eps^+,x_0,r)}{r^{N-1}} \right) &= \frac{2}{r^{N-1}} \int_{S_r(x_0)} w_\eps^+ \pa_\nu w_\eps^+  = \frac{2}{r^{N-1}} \int_{B_r(x_0)} |\nabla w_\eps^+|^2.
\end{align*}
Integrating in $(r_1,r_2)$, with $r_1$ and $r_2$ arbitrarily chosen, we infer that
\[
\frac{H(w_\eps^+,x_0,r_2)}{r_2^{N-1}}- \frac{H(w_\eps^+,x_0,r_1)}{r_1^{N-1}} = \int_{r_1}^{r_2} \left(\frac{2}{r^{N-1}} \int_{B_r(x_0)} |\nabla w_\eps^+|^2\right) \,dr,
\]
and passing to the limit as $\eps \to 0$ we deduce by monotone convergence that the same formula holds with $w^+$ instead of $w_\eps^+$. In the same way, we can prove that an analogue formula holds for $w^-$, and adding term by term the thesis follows.
\end{proof}

A similar statement holds for the functions $\tilde w_n$ under assumption (h2).

\begin{lemma}\label{lem: der H (h2)}
Assume that (h2) holds. Then
\[
\frac{d}{dr} H(\tilde w_n,x_0,r) =  \frac{N-1}{r}H(\tilde w_n,x_0,r) + 2 \int_{B_r(x_0)} \left(|\nabla \tilde w_n|^2 -k_n \tilde w_n \tilde g_n(\tilde w_n) \right),
\]
\end{lemma}
\begin{proof}
We refer to the forthcoming Proposition \ref{lem: def H 1}, which contains the thesis as a particular case. 
\end{proof}

The importance of the previous lemmas stays in the fact that they yields strong $H^1$ convergence of $\tilde w_n$ to $\tilde v$.

\begin{lemma}\label{lem: strong Hol1}
Up to a subsequence, $\tilde w_n \to \tilde v$ strongly in $H^1_{\loc}(\R^N)$.
\end{lemma}

\begin{proof}
We suppose at first that (h1) holds. Let $x_0 \in \R^N$, and let $0<r_1<r_2$. It is not difficult to compute the derivative of $H(\tilde w_n,x_0,r)$ using \eqref{eq w tilde 1} and the smoothness of $\tilde w_n$, deducing that
\[
\frac{H(\tilde w_n,x_0,r_2)}{r_2^{N-1}}- \frac{H(\tilde w_n,x_0,r_1)}{r_1^{N-1}} = \int_{r_1}^{r_2} \left(\frac{2}{r^{N-1}} \int_{B_r(x_0)} |\nabla \tilde w_n|^2 - k_n \tilde w_n \tilde  g_n(\tilde w_n) - \tilde f_n \tilde w_n\right)dr. 
\]
We pass to the limit as $n \to \infty$: recalling that $k_n \to 0$, that $\|\tilde f_n\|_{L^\infty(\Omega_n)} \to 0$, and that $\tilde w_n \to \tilde v$ locally uniformly, we have by definition ot $\tilde g_n$ that
\begin{equation}\label{18 1 10}
\frac{H(\tilde v,x_0,r_2)}{r_2^{N-1}}- \frac{H(\tilde v,x_0,r_1)}{r_1^{N-1}} = \lim_{n \to \infty} \int_{r_1}^{r_2} \left(\frac{2}{r^{N-1}} \int_{B_r(x_0)} |\nabla \tilde w_n|^2 \right)dr.
\end{equation}
On the other hand, by Lemmas \ref{lem: ar non 0} and \ref{lem: weak conv Hol1} the function $\tilde v$ fulfills the assumptions of Lemma \ref{lem: der H 2}, whence
\begin{equation}\label{18 1 11}
\frac{H(\tilde v,x_0,r_2)}{r_2^{N-1}}- \frac{H(\tilde v,x_0,r_1)}{r_1^{N-1}} = \int_{r_1}^{r_2} \left(\frac{2}{r^{N-1}} \int_{B_r(x_0)} |\nabla \tilde v|^2 \right)dr.
\end{equation}
Comparing \eqref{18 1 10} and \eqref{18 1 11}, we deduce that
\[
\lim_{n \to \infty} \int_{r_1}^{r_2} \left(\frac{2}{r^{N-1}} \int_{B_r(x_0)} |\nabla \tilde w_n|^2 \right)dr = \int_{r_1}^{r_2} \left(\frac{2}{r^{N-1}} \int_{B_r(x_0)} |\nabla \tilde v|^2 \right)dr.
\]
On the other hand, by weak convergence and using Fatou's lemma
\begin{align*}
\int_{r_1}^{r_2} \left(\frac{2}{r^{N-1}} \int_{B_r(x_0)} |\nabla \tilde v|^2 \right)dr &\le \int_{r_1}^{r_2} \left(\frac{2}{r^{N-1}} \liminf_{n \to \infty} \int_{B_r(x_0)} |\nabla \tilde w_n|^2 \right)dr \\
& = \int_{r_1}^{r_2} \liminf_{n \to \infty} \left(\frac{2}{r^{N-1}} \int_{B_r(x_0)} |\nabla \tilde w_n|^2  \right)dr \\
& \le \liminf_{n \to \infty} \int_{r_1}^{r_2} \left(\frac{2}{r^{N-1}} \int_{B_r(x_0)} |\nabla \tilde w_n|^2 \right)dr.
\end{align*}
It follows that all the previous inequalities must be equalities, and in particular
\[
\int_{r_1}^{r_2} \frac{2}{r^{N-1}}\left( \int_{B_r(x_0)} |\nabla \tilde v|^2- \liminf_{n \to \infty} \int_{B_r(x_0)} |\nabla \tilde w_n|^2 \right)dr = 0.
\]
Since the term inside the bracket is nonpositive for every $r \in (r_1,r_2)$, we finally deduce that for almost every $r \in (r_1,r_2)$
\[
\int_{B_r(x_0)} |\nabla \tilde v|^2 = \liminf_{n \to \infty} \int_{B_r(x_0)} |\nabla \tilde w_n|^2.
\]
If necessary extracting a further subsequence, we infer that $\|\tilde w_n\|_{H^1(B_r(x_0))} \to \|\tilde v\|_{H^1(B_r(x_0))}$ for almost every $r \in (r_1,r_2)$, for every $0<r_1<r_2$. Together with the local weak $H^1$ convergence and Sobolev embedding, this yields local strong $H^1$ convergence.

In case (h2) is in force, the expression of the derivative of $H(\tilde w_n,x_0, \cdot)$ is given by Lemma \ref{lem: der H (h2)}. From this, it is not difficult to derive \eqref{18 1 10}, and appealing to Lemma \ref{lem: der H 2} we also have \eqref{18 1 11}. At this point we can proceed as under assumption (h1).
\end{proof}

We are finally ready for the:

\begin{proof}[Conclusion of the proof of Proposition \ref{prop: Hol 1}]
Let (h1) holds, and let $Y \in C^\infty_c(\R^N,\R^N)$. For any $n$ sufficiently large, so that $\supp \, Y \subset \Omega_n$, we multiply the equation \eqref{eq w tilde} by $\langle\nabla \tilde w_n, Y\rangle$ and, integrating by parts, we find: 

\begin{equation}\label{var dom n}
\int_{\Omega_n} \langle dY [\nabla \tilde w_n], \nabla \tilde w_n\rangle - \tilde f_n\langle \nabla\tilde w_n,Y\rangle- \div Y \left[ \frac12 |\nabla \tilde w_n|^2 - k_n \tilde G_n(\tilde w_n) \right] = 0,
\end{equation}

where $\tilde G_n$ is a primitive of $\tilde g_n$:
\[
\tilde G_n(w):= \lambda_+\frac{\left( \tilde \eps_n^2 + (w^+)^2\right)^\frac{q}2}{q}+ \lambda_-\frac{\left( \tilde \eps_n^2 + (w^-)^2\right)^\frac{q}2}{q},
\]

In other words, the domain variation formula holds true for $\tilde w_n$.
Taking the limit as $n \to \infty$, we deduce by local uniform and strong $H^1_{\loc}$ convergence that
\begin{equation}\label{var dom limit}
\int_{\R^N} \langle dY [\nabla \tilde v], \nabla \tilde v\rangle- \frac12 \div Y   |\nabla \tilde v|^2 = 0,
\end{equation}
where we also used the fact that $\tilde f_n, k_n \to 0$. By density, \eqref{var dom limit} holds also for every compactly supported $Y \in \mathrm{Lip}(\R^N,\R^N)$.

If we are under assumption (h2), the validity of a formula of type \eqref{var dom n} follows by scaling \eqref{var dom} (which holds by assumption), and by local uniform and strong $H^1$ convergence we still obtain \eqref{var dom limit}. 

At this point, both under (h1) and (h2), let us fix $x_0 \in \R^N$ and $r>0$. For any $\delta>0$, we define $\varphi_{\delta}: [0,+\infty) \to [0,1]$ and $\eta_\delta: \R^N \to \R$ by
\[
\varphi_{\delta}(t):=     \begin{cases} 1 & \text{if $0 \le t \le r$} \\ 0 & \text{if $t \ge r+\delta$} \\ 1-\frac{1}{\delta}(t-r) & \text{if $r<t<r+\delta$}, \end{cases} \quad \text{and} \quad \eta_\delta(x):= \varphi_\delta(|x-x_0|).
\] 
Taking $Y=(x-x_0) \eta_{\delta}$ in \eqref{var dom limit}, we infer that
\[
\int_{B_{r+\delta}(x_0)} \frac{2-N}2|\nabla \tilde v|^2 \eta_\delta + \langle \nabla \eta_\delta, \nabla \tilde v \rangle \langle x-x_0, \nabla \tilde v \rangle - \frac{1}{2} |\nabla \tilde v|^2 \langle \nabla \eta_\delta, x-x_0 \rangle  =0.
\]
Passing to the limit as $\delta \to 0^+$, we finally obtain
\[
\begin{split}
\int_{S_r(x_0)} |\nabla \tilde v|^2 & = \frac{N-2}{r} \int_{B_r(x_0)} |\nabla \tilde v|^2 + 2\int_{S_r(x_0)} (\pa_\nu \tilde v)^2,
\end{split}
\]
for every $x_0 \in \R^N$, for every $r>0$. Recalling that $\tilde v$ is harmonic in $\{\tilde v \neq 0\}$, it results that $(\tilde v^+,\tilde v^-)$ fulfills the assumptions of \cite[Theorem 1.1]{TavTer} in $\R^N$, whence it follows that $\tilde v$ is harmonic in the whole space $\R^N$. But $\tilde v$ is also globally $\alpha$-H\"older continuous and non-constant, ad this is in contradiction with the Liouville theorem for harmonic functions.
\end{proof}

\section{H\"older estimates for the derivatives}

In this section we complete the proof of Theorem \ref{prop: bounds}. The strategy is similar to the one used for Proposition \ref{prop: Hol 1}. Let $\alpha \in (0,q/(2-q))$, let us fix a compact set $K \subset B_\rho$, and let $\eta \in C^\infty_c(B_\rho)$ with $\eta \equiv 1$ on $K$, $0 \le \eta \le 1$. If we show that the sequence $\{(\eta u_n): n \in \N\}$ is bounded in $C^{1,\alpha}(\overline{B_\rho})$, then, using the fact that $\eta \equiv 1$ on $K$, we infer that $\{u_n\}$ is bounded in $C^{1,\alpha}(K)$; since $K$ has been arbitrarily chosen, Theorem \ref{prop: bounds} follows. Thus, from now on we focus on the boundedness of $\{(\eta u_n)\}$. Let $\bar x \in B_\rho \setminus \supp\,\eta$; we have $(\eta u_n)(\bar x) = |\nabla (\eta u_n)(\bar x)| = 0$ for every $n$. Therefore, if we show that $[\pa_{x_i}(\eta u_n)]_{C^{0,\alpha}(\overline{B_\rho})}$ is bounded for every $i$, we obtain boundedness in $C^{1,\alpha}$, as desired. Suppose then by contradiction that 
\begin{align*}
L_n& := \max_{i =1,\dots, N} \, \sup_{\substack{x,y \in \overline{B_\rho} \\ x \neq y}} \frac{|\pa_{x_i} (\eta u_n)(x) - \pa_{x_i}(\eta u_n)(y)|}{|x-y|^\alpha}  \\
& =  \frac{|\pa_{x_{i_n}} (\eta u_n)(x_n) - \pa_{x_{i_n}}(\eta u_n)(y_n)|}{|x_n-y_n|^\alpha} \to +\infty.
\end{align*}
Of course, the sequences $\{x_n\}$ and $\{y_n\}$ are different than those of the previous section. Up to a subsequence and a relabelling, we can assume that $i_n=1$ for every $n$, and moreover that $\eta(x_n)>0$ for every $n$. We also set $r_n = |x_n-y_n|$. We have $0< r_n < \textrm{diam} B_\rho$, and hence up to a subsequence $r_n \to \bar r \ge 0$.

In order to reach a contradiction, we introduce again two blow-up sequences:
\[
\begin{split}
& v_n(x):= \frac{\eta(x_n+r_n x)}{L_n r_n^{1+\alpha}} \left(u_n(x_n+r_n x)-u_n(x_n)\right), \\
& w_n(x):= \frac{\eta(x_n)}{L_n r_n^{1+\alpha}} \left(u_n(x_n+r_n x)-u_n(x_n)\right),
\end{split}
\]
both defined on the closure of the scaled domains $(B_\rho-x_n)/r_n =: \Omega_n$. Notice that, since $x_n \in \supp \, \eta \Subset B_\rho$, it results that $\Omega_n$ converges to a limit domain $\Omega_\infty$, which contains a ball $B_{R}=B_R(0)$ for some $R>0$. Moreover, if we know that $r_n \to 0$, we can infer that $\Omega_\infty=\R^N$. 

In what follows we collect the basic properties of $v_n$ and $w_n$. Let $D$ be any compact set of $\Omega_\infty$, and let $x,y \in D$. Then $x,y \in \Omega_n$ for sufficiently large $n$, and
\begin{align*}
|\pa_{x_i} v_n(x) - \pa_{x_i} v_n(y)| &\le \frac{1}{L_n r_n^\alpha} |\pa_{x_i}(\eta u_n)(x_n+r_nx) - \pa_{x_i}(\eta u_n)(x_n + r_n y)| \\
& + \left| \frac{u_n(x_n)}{L_n r_n^\alpha}\right| |\pa_{x_i} \eta(x_n+r_n x) - \pa_{x_i} \eta(x_n+r_n y)| \\
& \le |x-y|^\alpha + \left| \frac{u_n(x_n)}{L_n}\right| r_n^{1-\alpha} \ell |x-y|.
\end{align*}
where we used that $\eta$ is Lipschitz continuous with constant $\ell$. Since $D$ is compact, $L_n \to +\infty$, $\{u_n\}$ is bounded in $L^\infty(\supp \, \eta)$ (Lemma \ref{lem: unif l infty}), and $0 < r_n \le \textrm{diam} B_\rho$ for every $n$, there exists $\bar n=\bar n(D)$ such that, if $n > \bar n$, we have
\[
\left| \frac{u_n(x_n)}{L_n}\right| r_n^{1-\alpha} \ell \sup_{x,y \in D} |x-y|^{1-\alpha} \le 1;
\]
thus, we proved that 
\begin{equation}\label{unif const}
\forall D \Subset \Omega_\infty \ \exists \bar n= \bar n(D): \ n > \bar n \quad \implies \quad \sup_{\substack{x,y \in D \\ x \neq y}} \frac{|\pa_{x_i} v_n(x) - \pa_{x_i} v_n(y)|}{|x-y|^\alpha} \le 2.
\end{equation}

Furthermore, we also note that
\begin{equation}\label{grad non cos}
\begin{split}
\left| \pa_{x_1} v_n(0)  - \pa_{x_1} v_n\left(\frac{y_n-x_n}{r_n}\right)  \right|&= \left|\frac{1}{L_n r_n^\alpha} \left( \pa_{x_1} (\eta u_n)(x_n) - \pa_{x_1} (\eta u_n)(y_n) \right) \right.\\
& \qquad + \left. \frac{u_n(x_n)}{L_n r_n^\alpha} \left( \pa_{x_1} \eta(x_n) - \pa_{x_1} \eta(y_n) \right) \right|\\ 
& = 1+ O \left( \frac{\ell |u_n(x_n)| r_n^{1-\alpha}}{L_n}\right) =  \left| \frac{y_n-x_n}{r_n} \right|^\alpha + o(1)
\end{split}
\end{equation}
as $n \to \infty$. Hence, on any compact set $D$ of $\Omega_\infty$ the quantities $[\pa_{x_i} v_n]_{C^{0,\alpha}(D)}$ are uniformly bounded from above, and moreover $[\pa_{x_1} v_n]_{C^{0,\alpha}(\Omega_n)}$ is uniformly bounded also from below. Let us then define
\begin{equation}\label{def bar v}
\bar v_n(x):= v_n(x)- \langle\nabla v_n(0), x\rangle, \quad x \in \Omega_n.
\end{equation}
For every $n$ it results that $\bar v_n(0) = |\nabla \bar v_n(0)| = 0$, and $[\pa_{x_i} v_n]_{C^{0,\alpha}(D)} = [\pa_{x_i} \bar v_n]_{C^{0,\alpha}(D)}$ for every $D \subset \Omega_n$ compact, for every $i=1,\dots,N$. Therefore, by Ascoli-Arzel\'a theorem $\bar v_n \to \bar v$ in $C^{1,\gamma}_{\loc}(\Omega_\infty)$ as $n \to \infty$ up to a subsequence, for $0<\gamma<\alpha$. Passing to the limit inside \eqref{unif const}, we have that $\bar v$ is of class $C^{1,\alpha}(\Omega_\infty)$, with $[\pa_{x_i} v]_{C^{0,\alpha}(D)} \le 2$ for every $D \Subset \Omega_\infty$. Notice also that $(y_n-x_n)/r_n \in \supp\, \eta(x_n+r_n \cdot\,) \cap S_1$, and hence up to a further subsequence $(y_n-x_n)/r_n \to z \in \Omega_\infty$. By local $C^{1}$ convergence, \eqref{grad non cos} entails then
\[
\pa_{x_1} \bar v(0) - \pa_{x_1} \bar v(z) = 1,
\]
that is, $\bar v$ has non-constant gradient.

\begin{lemma}
It results that $r_n \to 0$ as $n \to \infty$, so that $\Omega_\infty=\R^N$.
\end{lemma}

\begin{proof}
If $r_n \to \bar r > 0$, then by Lemma \ref{lem: unif l infty}
\[
\sup_{x \in \Omega_n} |v_n(x)| \le \frac{2 \|\eta\|_{L^\infty(B_\rho)} \|u_n\|_{L^\infty(\supp \eta)}}{ r_n^{1+\alpha} L_n} \le \frac{C}{ \bar r^{1+\alpha} L_n} \to 0
\]
as $n \to \infty$. In particular, $v_n \to 0$ uniformly on compact sets in $\Omega_\infty$. By definition \eqref{def bar v}, and taking the limit, we deduce that 
\begin{equation}\label{bar v uguale}
\bar v(x) = \lim_{n \to \infty} \langle\nabla v_n(0),x\rangle \qquad \text{for every $x \in \Omega_\infty$}.
\end{equation}
Recall now that $\Omega_\infty$ compactly contains a ball $B_R$ centered in the origin. If one component $\{\pa_{x_i} v_n(0)\}$ were unbounded, we would deduce from \eqref{bar v uguale} that
\[
|\bar v( R e_i )| = R \lim_{n \to \infty} |\langle\nabla v_n(0),e_i\rangle| = +\infty,
\] 
in contradiction with the fact that $\bar v \in C^{1,\alpha}(\overline{B_R})$. Therefore, $\{\nabla v_n(0)\}$ is bounded and up to a subsequence converges to a vector $\nu \in \R^N$, so that $\bar v(x) = \langle\nu, x\rangle$. This is in turn in contradiction with the fact that the gradient of $\bar v$ is not constant, and completes the proof.
\end{proof}

The main consequence is that $\Omega_\infty=\R^N$. Then, to sum up, so far we showed that $\bar v_n \to \bar v$ in $C^{1,\gamma}_{\loc}(\R^N)$ for $0<\gamma<\alpha$, and 
\begin{equation}\label{prop bar v}
[\pa_{x_i} \bar v]_{C^{0,\alpha}(\R^N)} \le 2 \quad \forall i, \text{ and } \exists z \in S_1 \text{ such that } \pa_{x_1} \bar v(0) - \pa_{x_1} \bar v(z) =1.
\end{equation}
Now we consider a new sequence
\[
\bar w_n(x):= w_n(x)- \langle\nabla w_n(0), x\rangle.
\]

\begin{lemma}
We have $\bar w_n \to \bar v$ locally uniformly in $\R^N$.
\end{lemma}

\begin{proof}
This follows from the fact that $\|\bar w_n-\bar v_n\|_{L^\infty(D)} \to 0$ as $n \to \infty$, for every $D \subset \R^N$ compact. Indeed, we have that
\[
\nabla v_n(0) = \frac{\eta(x_n)}{L_n r_n^\alpha} \nabla u_n(x_n) = \nabla w_n(0);
\]
hence, using the boundedness of $\{u_n\}$ in $C^{0,\alpha}(\supp \, \eta)$ (Proposition \ref{prop: Hol 1}), we infer that
\begin{align*}
|\bar w_n(x) - \bar v_n(x)| & = |w_n(x)-v_n(x)| \\
& = \frac{1}{L_n r_n^{1+\alpha}} |\eta(x_n+r_n x)-\eta(x_n)| |u_n(x_n+r_n x)- u_n(x_n)| \\
& \le \frac{C}{L_n r_n^{1+\alpha}}\cdot r_n |x| \cdot r_n^\alpha |x|^\alpha. 
\end{align*}
This quantity tends to $0$, as $n \to \infty$, uniformly on compact sets.
\end{proof}

The advantage of working with $\bar w_n$ instead of $\bar v_n$ stays in the fact that $\bar w_n$ solves a sufficiently nice equation. In case (h1) holds, let 
\[
a_n:= \frac{\eta(x_n) u_n(x_n)}{L_n r_n^{1+\alpha}}, \qquad b_n:= \nabla w_n(0), 
\] 
\[
\bar \eps_n:=  \frac{\eta(x_n) \eps_n}{L_n r_n^{1+\alpha}}, \qquad k_n:=   r_n^{q-\alpha(2-q)}  \left(\frac{\eta(x_n)}{L_n}\right)^{2-q},
\] 
\[
\bar g_n(\bar w_n) := \lambda_+ \frac{\left( \bar w_n + a_n + \langle b_n,x\rangle\right)^+}{\left( \bar \eps_n^2 + \left( \bar w_n + a_n + \langle b_n,x\rangle\right)^2\right)^\frac{2-q}2}-\lambda_- \frac{\left( \bar w_n + a_n + \langle b_n,x\rangle\right)^-}{\left( \bar \eps_n^2 + \left( \bar w_n + a_n + \langle b_n,x\rangle\right)^2\right)^\frac{2-q}2},
\]
\[
\bar f_n(x)= \frac{r_n^{1-\alpha} \eta(x_n)}{L_n} f_n(x_n+r_n x) .
\]
It is immediate to check that 
\begin{subequations}\label{eq bar w}
\begin{equation}\label{eq bar w 1}
-\Delta \bar w_n + \bar f_n= k_n \bar g_n(\bar w_n) \qquad \text{in $\Omega_n$}.
\end{equation} 
Analogously, in case (h2) holds, letting $a_n, b_n, k_n$ as above and
\[
\bar g_n(\bar w_n) = \alpha_n \left(\lambda_+ \left(\left( \bar w_n + a_n +\langle b_n,x\rangle\right)^+\right)^{q-1} -\lambda_-\left(\left( \bar w_n + a_n + \langle b_n,x\rangle\right)^-\right)^{q-1}\right),
\]
it results that
\begin{equation}\label{eq bar w 2}
-\Delta \bar w_n = k_n \bar g_n(\bar w_n) \qquad \text{in }\Omega_n \setminus \{\bar w_n + a_n + \langle b_n,x\rangle=0\}.
\end{equation}
\end{subequations}
In both cases, in order to reach a contradiction we have to discuss several possibilities related to the asymptotic behavior of the quantities $\bar \eps_n, a_n, b_n$. This is the content of the next lemmas, for which it will be useful to observe that, as $0<\alpha<q/(2-q)$, we have $k_n \to 0$, and $\|\bar f_n\|_{L^\infty(\Omega_n)} \to 0$, as $n \to \infty$.

\begin{lemma}\label{lem: partial divergence}
Both $\{a_n\}$ and $\{b_n\}$ are bounded sequences.
\end{lemma}

\begin{proof} 
For the sake of contradiction, assume not; then either $|a_n| \to +\infty$ and $|a_n|/|b_n| \to +\infty$, or $|b_n| \to +\infty$ and $|a_n|/|b_n| \le C$.

If $|a_n| \to +\infty$ and $|a_n|/|b_n| \to +\infty$, then $|a_n + \langle b_n,x\rangle| \to +\infty$ uniformly on compact sets of $\R^N$, and the nodal set $\{a_n + \langle b_n,x\rangle=0\}$ escapes at infinity. For concreteness, assume that $a_n + \langle b_n,x\rangle\to +\infty$ locally uniformly. Let $D$ be a compact set of $\R^N$. Since $\bar w_n$ is locally uniformly convergent, we have that $\bar w_n + a_n + \langle b_n,x\rangle>0$ definitely on $D$, and hence by \eqref{eq bar w} we deduce that
\[
|\Delta \bar w_n| \le k_n \lambda_+  \left(\bar w_n + a_n + \langle b_n,x\rangle\right)^{q-1} + o(1) \to 0
\]
uniformly on $D$. Since $D$ has been arbitrarily chosen, the limit $\bar v$ of $\bar w_n$ is a harmonic function in $\R^N$, of class $C^{1,\alpha}$, with bounded $[\nabla \bar v]_{C^{0,\alpha}(\R^N)}$ seminorm (thus $\bar v$ has strictly subquadratic growth at infinity), and with non-constant gradient. This is in contradiction with the Liouville theorem for harmonic functions.

We suppose now that $|b_n| \to +\infty$ and $|a_n|/|b_n| \le C$, and analyze the convergence of the nodal set $a_n +\langle b_n,x\rangle = 0$. We have
\[
a_n + \langle b_n,x\rangle=0 \quad \iff \quad \frac{a_n}{|b_n|} + \langle\frac{b_n}{|b_n|},x\rangle = 0,
\]
and passing to the limit the nodal set is described by $\mu + \langle\nu, x\rangle=0$ for some $\mu \in \R$, and $\nu \in S_1$. This means that $a_n+ \langle b_n,x\rangle \to +\infty$ on one of the half spaces $\Sigma_1:= \{\mu + \langle \nu,x\rangle >0\}$ or $\Sigma_2:=\{\mu + \langle \nu,x\rangle <0\}$, and $a_n+ \langle b_n,x\rangle \to -\infty$ on the other. 

Let $x \in \Sigma_1 \cup \Sigma_2$. Then $|\bar w_n(x) + a_n+ \langle b_n,x\rangle| \to +\infty$, and hence, passing to the limit in \eqref{eq bar w} (as we did in Lemma \ref{lem: ar non 0}), we deduce that $\bar v$ is harmonic in $\{\mu + \langle \nu,x\rangle \neq 0\}$. That is, $\bar v$ is a harmonic function in $\R^N$ minus a hyperplane, and moreover is of class $C^1$ in $\R^N$. It follows that $\bar v$ is harmonic in $\R^N$, and a contradiction can be reached as above.
\end{proof}

Lemma \ref{lem: partial divergence} shows that, up to a subsequence, $a_n \to a \in \R$ and $b_n \to b \in \R^N$. 

\begin{lemma}
In case (h1) holds, it results that $\bar \eps_n \to 0$ as $n \to \infty$.
\end{lemma}

\begin{proof}
If up to a subsequence $\bar \eps_n \to \bar \eps >0$, then by \eqref{eq bar w 1} and the boundedness of $\bar w_n$, $a_n$, and $b_n$, we have
\[
|\Delta w_n| \le C\frac{k_n}{\bar \eps_n^{2-q}} |\bar w_n + a_n + \langle b_n,x\rangle| + o(1) \to 0
\]
as $n \to \infty$, uniformly on compact sets. Therefore, $\bar v$ is harmonic in $\R^N$, and exactly as in the previous cases we reach a contradiction.
\end{proof}

We showed that 
\[
w_n = \bar w_n  + \langle b_n,x\rangle \to \bar v + \langle b,x\rangle =: w
\]
(recall the definition of $\bar w_n$ and of $b_n$) locally uniformly on compact sets of $\R^N$, and $\bar \eps_n \to 0$. Moreover,
\[
-\Delta w_n = -\Delta \bar w_n = o(1) + k_n \bar g_n(\bar w_n) \qquad \text{in }\Omega_n \setminus \{w_n +a_n = 0\}.
\]
Since $w$ is continuous and the convergence takes place locally uniformly, and recalling that $k_n \to 0$, passing to the limit as $n \to \infty$ as in Lemma \ref{lem: ar non 0} we obtain that $w$ is a harmonic function in the open set $\R^N \setminus \{w + a=0\}$. 

\begin{lemma}
Up to a subsequence, we have that $w_n \to w$ strongly in $H^1_{\loc}(\R^N)$.
\end{lemma}

\begin{proof}
It is possible to proceed exactly as in Lemma \ref{lem: weak conv Hol1}, obtaining weak $H^1_{\loc}$ convergence. At this point, for the strong convergence, we can proceed as in Lemmas \ref{lem: strong Hol1}.
\end{proof}

\begin{proof}[Conclusion of the proof of Theorem \ref{prop: bounds}]
The conclusion is very similar to the one of Proposition \ref{prop: Hol 1}. We write down the domain variation formula for $w_n$, take $Y = (x-x_0) \eta_\delta$, and pass to the limit in $n$. By strong $H^1_{\loc}$ convergence, we deduce that both under (h1) and (h2)
\[
\begin{split}
\int_{S_r(x_0)} |\nabla w|^2 & = \frac{N-2}{r} \int_{B_r(x_0)} |\nabla w|^2 + 2\int_{S_r(x_0)} (\pa_\nu w)^2,
\end{split}
\]
for every $x_0 \in \R^N$ and $r>0$. Since $w$ is harmonic in $\R^N \setminus \{w + a=0\}$, Theorem 1.1 in \cite{TavTer} implies that $w$ is harmonic in the whole space $\R^N$. But recalling \eqref{prop bar v} and the definition of $w$, we have that $w$ is a harmonic function in $\R^N$, of class $C^{1,\alpha}$, with bounded $[\nabla w]_{C^{0,\alpha}(\R^N)}$ seminorm (thus $w$ has strictly subquadratic growth at infinity), and with non-constant gradient. This is once again in contradiction with the Liouville theorem for harmonic functions.
\end{proof}

As a straightforward corollary:

\begin{proof}[Conclusion of the proof of Theorem \ref{thm: reg good}]
It remains only to show that good solutions satisfy the domain variation formula \eqref{var dom}. This can be obtained simply writing the analogue identity for the approximating functions $u_n$, as in the conclusion of the proof of Proposition \ref{prop: Hol 1}, and taking the limit as $n \to \infty$.
\end{proof}

\section{Study of the nodal set}\label{sec: changes}

As already announced, we refer to \cite{SoTe2018} for most of the results, emphasizing here only the necessary modifications.

Before proceeding, we mention that we shall prove slightly stronger results than Theorems \ref{thm: very strong V}, \ref{thm: blow-up} and \ref{thm: Hausdorff}. Indeed, we shall prove the validity of such theorems for $C^{1,{\frac{q}{2-q}}-}$ stationary solutions to \eqref{eq}, rather than for good solutions. Of course, the above statements follow since good solutions are regular stationary solutions, by Theorem \ref{thm: reg good}. 

\subsection{Preliminaries}

The following statement corresponds to \cite[Proposition 2.1]{SoTe2018}, and follows as direct consequence of the regularity.
\begin{proposition}\label{thm: blow-up pre}
Let $u$ be a stationary solution of \eqref{eq}, of class $C^{1,{\frac{q}{2-q}}-}$, $x_0 \in Z(u)$, and $R  \in (0, \dist(x_0,\pa B_1))$. Then the following alternative holds:
\begin{itemize}
\item[($i$)] either $\nabla u(x_0)\neq 0$, and therefore there exists a function $\Gamma_{x_0}$ such that  
\[
u(x) = \langle\nabla u(x_0),x-x_0\rangle + \Gamma_{x_0}(x) \qquad \text{in $B_R(x_0)$},
\]
with 
\[
\begin{cases} \Gamma_{x_0}(x)| \le C |x-x_0|^{1+\alpha}  \\ 
|\nabla \Gamma_{x_0}(x)| \le C |x-x_0|^{\alpha} \end{cases}
 \qquad \text{in $B_R(x_0)$}
\]
for a suitable constant $C$ depending on $\alpha \in (0,q/(2-q))$;
\item[($ii$)] or, for every $\eps>0$, there exists $C_\eps>0$ such that
\[
\begin{cases}
|u(x)| \le C_\eps |x-x_0|^{\frac2{2-q}-\eps} \\    |\nabla u(x)| \le C_\eps |x-x_0|^{\frac{q}{2-q}-\eps}\end{cases} \qquad \text{in $B_R(x_0)$}.
\]
\end{itemize}
\end{proposition}

\subsection{Almgren and Weiss type functionals}\label{sub: weiss}

Let $\Omega \subset \R^N$ be a domain, $\mu_+, \mu_-> 0$, and let $v$ be a stationary solution of class $C^{1,\alpha}$ (for some $0< \alpha <q/(2-q)$) to 
\begin{equation}\label{eq mu}
-\Delta v = \mu_+ (v^+)^{q-1} -\mu_- (v^-)^{q-1} \quad \text{in $\Omega$}, 
\end{equation}
with $q \in (0,1)$. Let 
\[
F_{\mu_+, \mu_-}(v):= \mu_+ (v^+)^{q} +\mu_- (v^-)^{q}.
\]
For $x_0 \in \Omega$, $0<r<\dist(x_0,\pa \Omega)$, and for $\gamma,t>0$, we consider the functionals:
\begin{align*}
H(v,x_0, r) &:= \int_{S_r(x_0)} v^2,  \\
D_t(v,x_0,r) &:= \int_{B_r(x_0)}\left(|\nabla v|^2 - \frac{t}q F_{\mu_+, \mu_-}(v)\right), \\
N_t(v,x_0,r) &:= \frac{r D_t(v,x_0,r)}{H(v,x_0,r)}, \quad \text{defined provided that $H(v,x_0,r) \neq 0$},\\
W_{\gamma,t}(v,x_0,r)&:= \frac{1}{r^{N-2+2\gamma}} D_t(v,x_0,r) - \frac{\gamma}{r^{N-1+2\gamma}}H(v,x_0,r).
\end{align*}
In what follows, we fix $v$, solution to \eqref{eq mu}, and we derive several relations involving the above quantities and their derivatives. 

By definition and using the divergence theorem\footnote{Since $q \in (0,1)$ the classical divergence theorem is not applicable, but by $C^{1,\alpha}$ regularity we can appeal to more general versions such as \cite[Proposition 2.7]{HoMiTa}.}, we have
\[
D_t(v,x_0,r) = \int_{S_r(x_0)} v \, \pa_\nu v - \frac{t-q}{q} \int_{B_r(x_0)} F_{\mu_+, \mu_-}(v),
\]
\[
W_{\gamma,t}(v,x_0,r) = \frac{H(v,x_0,r)}{r^{N-1+2\gamma}} \Big( N_t(v,x_0,r)-\gamma \Big)
\]
whenever the right hand side makes sense.

\begin{proposition}\label{lem: def H 1}
Denoting with $\prime$ the derivative with respect to $r$, we have
\begin{equation}\label{der H}
\begin{split}
H'(v,x_0,r) &=  \frac{N-1}{r}H(v,x_0,r) + 2 \int_{S_r(x_0)} v \, \pa_\nu v \\
& =  \frac{N-1}{r}H(v,x_0,r) + 2 D_q(v,x_0,r),
\end{split}
\end{equation}
and 
\begin{equation}\label{der D}
\begin{split}
D_t'(v,x_0,r) & = \frac{N-2}{r} D_t(v,x_0,r) - \frac{2N-(N-2)t }{q r} \int_{B_r(x_0)} F_{\mu_+, \mu_-}(v) \\
& + \int_{S_r(x_0)}\left(2 v_\nu^2 + \left(\frac{2-t}{q}\right) F_{\mu_+, \mu_-}(v)\right).
\end{split}
\end{equation}
\end{proposition}
\begin{proof}
With respect to the case $q \in [1,2)$, dealing with $q \in (0,1)$ requires some extra care. The validity of the first equality in \eqref{der H} is obvious, since $v \in C^{1,\alpha}(\overline{B_r(x_0)})$ by Theorem \ref{thm: reg good}. The second equality follows from the divergence theorem in \cite[Proposition 2.7]{HoMiTa}, which is applicable since $v \pa_\nu v \in C(\overline{B_r(x_0)})$ and $\div(v \nabla v) = |\nabla v|^2 + F_{\mu_+,\mu_-}(v) \in L^1(B_r(x_0))$. 

Let us focus now on \eqref{der D}. We proceed exactly as in the conclusion of the proof of Proposition \ref{prop: Hol 1}. Taking $Y=(x-x_0) \eta_{\delta}$ in the variation of domains formula \eqref{var dom}, and passing to the limit as $\delta \to 0^+$, we obtain
\begin{equation}\label{loc poh}
\begin{split}
\int_{S_r(x_0)} |\nabla v|^2 & = \frac{N-2}{r} \int_{B_r(x_0)} |\nabla v|^2 - \frac{ 2N}{q r} \int_{B_r(x_0)} F_{\mu_+, \mu_-}(v) \\
& \qquad + \int_{S_r(x_0)}\left(2 (\pa_\nu v)^2 + \frac{2}{q} F_{\mu_+, \mu_-}(v)\right).
\end{split}
\end{equation}
At this point \eqref{der D} follows easily.
\end{proof}

With \eqref{der H} and \eqref{der D} in our hands the expressions of the derivative $W_{\gamma,t}$ can be established exactly as in \cite[Proposition 2.3 and Corollary 2.4]{SoTe2018}.

%
%

\subsection{A unique continuation property} We conclude this section proving the validity of the classical unique continuation property for good solutions to \eqref{eq}. We follow the same strategy introduced in \cite{SoWe}, which concern unique continuation for sublinear equations.

\begin{theorem}\label{thm: unique cont class}
Let $q \in (0,1)$, $\lambda_+,\lambda_->0$, and let $u$ be a stationary solution to \eqref{eq} in $B_1$, of class $C^{1,{\frac{q}{2-q}}-}(B_1)$. If $u \equiv 0$ in $B_r(x_0)$ for some $x_0 \in B_1$ and $r >0$, then $u \equiv 0$ in $B_1$. 
\end{theorem}

\begin{proof}
The derivatives of $H(u,x_0,\cdot)$ and $D_q(v,x_0,\cdot)$ are given by \eqref{der H} and \eqref{der D}, respectively. As a consequence, we can repeat the proof of \cite[Theorem 1.1]{SoWe} essentially word by word, deducing the thesis.
\end{proof}

\subsection{Universal bound on the vanishing order and non-degeneracy}\label{sec: strong and non-deg}

At a first stage, it is convenient to work with a different notion of vanishing order, which involves the norm $\|\cdot\|_{x_0,r}$ defined by \eqref{def norm}.

\begin{definition}\label{def: order}
For a solution $u$ to \eqref{eq}, let $x_0 \in Z(u)$. The $H^1$-\emph{vanishing order of $u$ in $x_0$} is defined as the number $\cO(u,x_0) \in \R^+$ with the property that 
\[
\limsup_{r \to 0^+} \frac1{r^{2\beta}} \|u\|_{x_0,r}^2  = \begin{cases} 0 & \text{if $0 <\beta< \cO(u,x_0)$} \\
+\infty  & \text{if $\beta > \cO(u,x_0)$}.
\end{cases}
\]  
If no such number exists, then 
\[
\limsup_{r \to 0^+} \frac1{r^{2 \beta}} \|u\|_{x_0,r}^2 = 0 \quad \text{for any $\beta>0$},
\]
and we set $\cO(u,x_0)=+\infty$. 
\end{definition}
As in \cite{SoTe2018}, we prove the following variant of Theorems \ref{thm: very strong V}.

\begin{theorem}\label{thm: very strong}
Let $0<  q <1$, $\lambda_+,\lambda_->0$, $0 \not \equiv u$ be stationary solution to \eqref{eq}, of class $C^{1,{\frac{q}{2-q}}-}$, and let $x_0 \in Z(u)$. Then 
\[
\cO(u,x_0) \in \left\{1,\frac2{2-q} \right\}, \quad \text{and} \quad \liminf_{r \to 0^+} \frac{\|u\|_{x_0,r}^2}{r^{2 \cO(u,x_0)}}  >0.
\]
\end{theorem}

For this result, we can proceed precisely as in the proof of \cite[Theorems 4.1 and 4.2]{SoTe2018} (see also the proof of \cite[Theorem 1.2]{SoTe2018}). Having in our hands \eqref{der H}, \eqref{der D}, and the derivatives of $W_{\gamma,t}$, the proofs do not require any change at all.

\subsection{Blow-up limits} 

In this subsection we prove Theorems \ref{thm: very strong V} and \ref{thm: blow-up}. As an intermediate step, we formulate a variant of Theorem \ref{thm: blow-up} in term of the vanishing order $\cO$.

\begin{theorem}\label{thm: blow-up '}
Let $0< q<1$, $\lambda_+,\lambda_->0$, and let $0 \not \equiv u \in C^{1,{\frac{q}{2-q}}-}_{\loc}(B_1)$ be a stationary solution to \eqref{eq}. If $x_0 \in Z(u)$ with $\cO(u,x_0) = 2/(2-q)$, then for every sequence $0<r_n \to 0^+$ we have, up to a subsequence,
\[
\frac{u(x_0+r_n x)}{\|u\|_{x_0,r_n}}
 \to \bar u, \qquad \text{in $C^{1,\alpha}_{\loc}(\R^N)$, for every }0<\alpha<\frac{q}{2-q},
\]
where $\bar u \in C^{1,{\frac{q}{2-q}}-}(\R^N)$ is a non-trivial $2/(2-q)$-homogeneous stationary solution to
\begin{equation}\label{eq limit 1 12}
-\Delta \bar u= \mu \left( \lambda_+ (\bar u^+)^{q-1} -\lambda_- (\bar u^-)^{q-1} \right) \quad \text{in $\R^N$}
\end{equation}
for some $\mu > 0$.
\end{theorem}

Throughout this section the value $2/(2-q)$ will be denoted by $\gamma_q$. For $0<r<R <\dist(x_0,\pa B_1)$, we consider
\[
v_r(x):= \frac{u(x_0+rx)}{\|u\|_{x_0,r}} \quad \implies \quad \|v_r\|_{0,1} = 1.
\]
Then $v_r$ is a good solution of
\[
-\Delta v_r = 
 \left( \frac{r^{\gamma_q}}{\|u\|_{x_0,r}} \right)^{\frac{2}{\gamma_q}} \Big( \lambda_+ (v_r^+)^{q-1} - \lambda_- (v_r^-)^{q-1} \Big)\quad \text{in } \frac{1}{r}B_R.
\]
Notice that the scaled domains exhaust $\R^N$ as $r \to 0^+$, and that there exists $C>0$ such that
\[
\alpha_r := \left( \frac{r^{\gamma_q}}{\|u\|_{x_0,r}} \right)^{\frac{2}{\gamma_q}} \le C \qquad \text{for every $0<r<R$},
\]
by non-degeneracy.

\begin{lemma}\label{lem: weak to strong}
Let $q \in (0,1)$ and $\rho>0$. Let $0<r_n \to 0^+$ be such that $v_{r_n} \weak v$ weakly in $H^1(B_\rho)$. Then, up to a subsequence, we have that $v_{r_n} \to v$ in $C^{1,\alpha}_{\loc}(B_\rho)$ for every $0<\alpha<\min\{q/(2-q),1\}$.
\end{lemma}

\begin{proof}
The convergence follows directly by Theorem \ref{prop: bounds}, see also Remark \ref{rmk: on bounds (0,1)}.
\end{proof}

The next result corresponds to \cite[Lemma 6.3]{SoTe2018}.

\begin{lemma}\label{lem: v_r bdd}
Let $\rho>1$ be fixed. There exists $ r_\rho>0$ small enough such that the family $\{v_r: r \in (0,r_\rho)\}$ is bounded in $H^1(B_\rho)$.
\end{lemma}
\begin{proof}
It results that $\|v_r\|_{0,\rho} = \|u\|_{x_0,\rho r}/\|u\|_{x_0,r}$, and hence the thesis follows if there exist $r_\rho>0$ and a constant $C_\rho>0$ such that
\[
\frac{\|u\|_{x_0,\rho r}}{ \|u\|_{x_0,r}} \le C_\rho, \quad \text{for every $0<r<r_\rho$}.
\]
Let us suppose by contradiction that for a sequence $0<r_n \to 0^+$ it results 
\begin{equation}\label{30 11 1}
\frac{\|u\|_{x_0,\rho r_n}}{ \|u\|_{x_0,r_n}} \to +\infty \quad \text{as $n \to \infty$}.
\end{equation}
In such case, proceeding as in \cite[Lemma 6.3]{SoTe2018}, we can show that
\begin{equation}\label{30 11 2}
\frac{\|u\|_{x_0,\rho r_n}}{ (\rho r_n)^{\gamma_q}} \to +\infty \quad \text{as $n \to \infty$},
\end{equation}
and moreover
\[
\frac{H(u,x_0, \rho r_n)}{(\rho r_n)^{N-1}\|u\|_{x_0, \rho r_n}^2 } \ge \frac1{2(\gamma_q+1)} >0
\]
for every $n$ large. Now, by Lemma \ref{lem: weak to strong}, we have that $v_{r_n} \to \bar v$ in $C^{1,\alpha}_{\loc}(B_1)$, with $H(\bar v,0,1) >0$ and hence $\bar v \not \equiv 0$ in $B_1$. Moreover, $\bar v \equiv 0$ in $B_{1/\rho}$ by \eqref{30 11 1}, and $\alpha_{r_n} \to 0$ by \eqref{30 11 2}. If the limit $\bar v$ is harmonic, this is in contradiction with the unique continuation property. This is what happen in the sublinear case $q \in [1,2)$, see \cite{SoTe2018}. The problem is that, due to the singularity in the equations, the harmonicity of $\bar v$ is not evident, even with $\alpha_{r_n} \to 0$, since we cannot pass to the limit in the equation. We then appeal to the variation of domains formula \eqref{var dom}: we know that \eqref{loc poh} holds for $u$ (with $\mu_\pm=\lambda_\pm$, for every $r \in (0,1)$). Scaling it, we find that for every $r \in (0,1)$, and for every $n$,
\[
\begin{split}
\int_{S_r} |\nabla v_{r_n}|^2 & = \frac{N-2}{r} \int_{B_r} |\nabla v_{r_n}|^2 - \frac{ 2N \alpha_{r_n}}{q r} \int_{B_r} F_{\lambda_+, \lambda_-}(v_{r_n}) \\
& \qquad + \int_{S_r}\left(2 (\pa_\nu v_{r_n})^2 + \frac{2 \alpha_{r_n}}{q} F_{\lambda_+, \lambda_-}(v_{r_n})\right).
\end{split}
\]
Passing to the limit as $n \to \infty$, we obtain by $C^{1,\alpha}_{\loc}(B_1)$ convergence that
\[
\int_{S_r} |\nabla \bar v|^2 = \frac{N-2}{r} \int_{B_r} |\nabla \bar v|^2 + 2 \int_{S_r}(\pa_\nu \bar v)^2,
\]
for every $r \in (0,1)$. Moreover, by $C^{1,\alpha}$ convergence we have that $\bar v$ is harmonic in the open set $\{\bar v \neq 0\}$. Therefore, the pair $(\bar v^+, \bar v^-)$ fulfills the assumption of \cite[Theorem 1.1]{TavTer}, whence it follows that in the end, even if the equation is singular, the limit $\bar v$ must be harmonic in the whole $B_1$. As observed, this gives the desired contradiction.
\end{proof}

We can now complete the proof of Theorem \ref{thm: blow-up '}.

\begin{proof}[Proof of Theorem \ref{thm: blow-up '}]
Due to Lemma \ref{lem: v_r bdd}, the sequence $\{v_{r_n}\}$ is bounded in $H^1_{\loc}(\R^N)$; thus, compactness argument and Lemma \ref{lem: weak to strong}, together with a diagonal selection, imply that up to a subsequence $v_{r_n} \to v$ strongly in $C^{1,\alpha}_{\loc}(\R^N)$, for every $0<\alpha<q/(2-q)$. By non-degeneracy, up to a subsequence we have two possibilities:
\[
\text{either} \quad \frac{\|u\|_{x_0,r_n}}{r_n^{\gamma_q}} \to \ell  \in (0,+\infty), \quad \text{or} \quad \frac{\|u\|_{x_0,r_n}}{r_n^{\gamma_q}} \to +\infty.
\]
If $\|u\|_{r_n}/r_n^{\gamma_q} \to \ell$ finite, then by Theorem \ref{prop: bounds} and by convergence we have that $v$ is a $C^{1,{\frac{q}{2-q}}-}(\R^N)$ stationary solution to \eqref{eq limit 1 12}. The homogeneity can be proved following word by word the argument in \cite[Theorem 6.1]{SoTe2018}.

It remains to study the case $\|u\|_{r_n}/r_n^{\gamma_q} \to +\infty$. This means that $\alpha_{r_n} \to 0$ as $n \to \infty$, so that, proceeding as in Lemma \ref{lem: v_r bdd}, we can show that $v$ is harmonic: first, by scaling identity \eqref{loc poh} and passing to the limit, we deduce that
\[
\int_{S_r} |\nabla v|^2 = \frac{N-2}{r} \int_{B_r} |\nabla v|^2 + 2 \int_{S_r}(\pa_\nu v)^2,
\]
for every $r >0$. Also, $v$ is harmonic in $\R^N \setminus \{v = 0\}$, and hence $(v^+,v^-)$ fulfills the assumptions of \cite[Theorem 1.1]{TavTer}. It follows that $v$ is harmonic in $\R^N$. Moreover, $\cO(v,0) >1$ by $C^{1,\alpha}$ convergence, and hence by \cite[Lemma 4.1]{Wei01} (and the fact that harmonic functions always vanish with integer order) we have that
\[
\frac{1}{t^{N-2}}\int_{B_t} |\nabla v|^2 \ge \frac{2}{t^{N-1}} \int_{S_t} v^2
\]
for every $t >0$. On the other hand, arguing as in \cite[Theorem 6.1]{SoTe2018}, we also find
\[
\frac{1}{t^{N-2}}\int_{B_t} |\nabla v|^2 \le \frac{\gamma_q}{t^{N-1}} \int_{S_t} v^2
\]
for $t>0$, which is a contradiction since $\gamma_q<2$ and $v$ is nontrivial. This means that $\|u\|_{r_n}/r_n^{\gamma_q} \to +\infty$ can never take place.
\end{proof}

Theorems \ref{thm: very strong V} and \ref{thm: blow-up} follow by Theorem \ref{thm: blow-up '} exactly as in \cite{SoTe2018} Theorems 1.3, 1.4 and 1.6 follow by Theorem 6.1. 

\subsection{Hausdorff dimension of the singular set}

\begin{proof}[Proof of Theorem \ref{thm: Hausdorff}]
Let 
\[
\cF:=\left\{ v \in C^{1,{\frac{q}{2-q}}-}(\R^N)\left| \begin{array}{l} \text{$v$ is a stationary solution to} \\
-\Delta v = \mu \Big(\lambda_+ (v^+)^{q-1}- \lambda_-(v^-)^{q-1}\Big) \quad \text{in $B_\rho$} \\ \text{for some $\rho > 2$ and some $\mu > 0$}. \end{array}\right.\right\},
\]
endowed with $C^{1,\alpha}_{\loc}$ convergence (for some $0<\alpha<q/(2-q)$), let $\cC$ be the class of all the relatively closed subsets of $B_1$, and let us define the map $\Sigma: \cF \to \cC$ by 
\[
\Sigma(v) := \{x \in B_1: v(x) =  |\nabla v(x)| = 0\}.
\]
The family $\cF$ is closed under scalings and translations. The existence of a homogeneous blow-up within the class $\cF$ follows by Theorem \ref{thm: blow-up}, as well as the singular set assumption (here we use the $C^{1,\alpha}_{\loc}$ convergence of the blow-ups). Thus, the Federer reduction principle (for which we refer to \cite{Che,Sim,TavTer}) is applicable, and implies that there exists an integer $0 \le d \le N-1$ such that $\Sigma(v) \le d$ for every $v \in \cF$. Moreover, there exists a $d$-dimensional linear subspace $E \subset \R^N$, and a $\gamma$-homogeneous function $v \in \cF$ (for some $\gamma>0$) such that $\{v=|\nabla v|=0\}=E$, and 
\begin{equation}\label{d-dim}
\frac{v(x_0+\lambda x)}{\lambda^\gamma} = v(x) \qquad \text{for every $x_0 \in E$ and $\lambda>0$};
\end{equation}
identity \eqref{d-dim} means that $v$ is $\gamma$-homogeneous with respect to all the points in $E$, and hence, up to a rotation, it depends only on $N-d$ variables. Let us suppose by contradiction that $d=N-1$. Then, without loss of generality, we can suppose that $E=\{x_1=0\}$ and $v(x_1,\dots,x_N) = w(x_1)$ for a function $w \in C^{1,{\frac{q}{2-q}}-}(\R)$. It is clear that 
\[
-w'' = \mu \big(\lambda_+ (w^+)^{q-1}- \lambda_-(w^-)^{q-1}\big) \qquad \text{in $\R \setminus \{w=0\}$} 
\]
with $w(0) = w'(0)=0$. Moreover, it is not difficult to prove that $w$ satisfies the variation of domains formula
\[
\int_{\R} \varphi' \left( \frac{1}{2} (w')^2 - \frac{F_{\mu \lambda_+, \mu \lambda_-}(w)}{q} \right) = 0
\]
for every $\varphi \in C^\infty_c(\R)$, we refer to \cite[Lemma 5.9]{DaWaWe} for details. This means precisely that the Hamiltonian function  
\[
\cH(w,w') = \frac12 |w'|^2 + \frac{\mu \lambda_+}q (w^+)^q + \frac{\mu \lambda_-}{q}(w^-)^q
\]
is constant along $w$. Since the the only level curve of $\cH$ crossing the origin of the phase plane is the constant trajectory $0$, we deduce that necessarily $w \equiv 0$, in contradiction with the fact that $v \not \equiv 0$. The discreteness of $\Sigma(u)$ in dimension $N=2$ can be proved combining the above argument with the proof of \cite[Theorem 1.7 - $N=2$]{SoTe2018}. 
\end{proof}

\section{Local minimizers are good solutions}\label{sec: min are good}

The purpose of this final sections consists in showing that the class of good solutions is sufficiently large, and for instance includes local minimizers for the functional associated with \eqref{eq} with a fixed trace. Let then $u_0 \in H^1(B_1) \cap C^0(\overline{B_1})$, $\lambda_+,\lambda_->0$, and $q \in (0,1)$. We consider the minimization problem
\[
\inf\left\{J(u) : u \in \cK\right\},
\]
where $\cK:= \left\{ u \in H^1(B_1): u-u_0 \in H_0^1(B_1)\right\}$, and 
\[
J(u) := \int_{B_1} \left( \frac12|\nabla u|^2 - \frac{\lambda_+}{q} (u^+)^q - \frac{\lambda_-}{q} (u^-)^q \right).
\]

\begin{proposition}
There exists a minimizer $\bar u$ of $J$ on $\mathcal{K}$.
\end{proposition}

\begin{proof}
If $u \in \mathcal{K}$, then by the Poincar\'e inequality
\begin{equation}\label{18 1 1}
\begin{split}
J(u) &\ge \int_{B_1} |\nabla u|^2 - C \int_{B_1} |u|^q  \ge \int_{B_1} |\nabla u|^2 - C\left(\int_{B_1} u^2\right)^{\frac{q}2} \\
& \ge \int_{B_1} |\nabla u|^2- C \left(\int_{S_1} u^2 + \int_{B_1} |\nabla u|^2 \right)^\frac{q}2 \\
& \ge \int_{B_1} |\nabla u|^2 - C \left(\int_{B_1} |\nabla u|^2 \right)^\frac{q}2 - C.
\end{split}
\end{equation}
Since $q < 2$, this shows that $J$ is coercive and bounded from below on the closed convex set $\cK$. Therefore, the existence of a minimizer is given by the direct method of the calculus of variations.
\end{proof}

Now we prove that:

\begin{proposition}\label{prop: min are good}
In the above setting, $\bar u$ is a good solution.
\end{proposition}

\begin{proof}
Let $0<\eps_n \to 0^+$, let 
\[
G_n(u):= \frac{\lambda_+}q\left( \eps_n^2+(u^+)^2\right)^\frac{q}{2} + \frac{\lambda_-}q \left(\eps_n^2 +(u^-)^2\right)^\frac{q}{2}, \quad G(u) := \frac{\lambda_+}q (u^+)^q - \frac{\lambda_-}q (u^-)^q,
\]
and let us consider the penalized functional
\[
J_n(u) := \int_{B_1} \left( \frac{1}{2}|\nabla u|^2- G_n(u)\right) + \int_{B_1} \arctan\left( (u - \bar u)^2\right).
\]
For every $n$, and every $u \in \cK$, we have as in \eqref{18 1 1}
\[
\begin{split}
J_n(u) & \ge \int_{B_1} |\nabla u|^2 - C \left(\int_{B_1} |\nabla u|^2 \right)^\frac{q}2 - C,
\end{split}
\]
with constants $C>0$ independent of $n$. Thus, each $J_n$ is coercive and bounded from below, and hence admits a minimizer $u_n$ on $\cK$. Each $u_n$ solves
\[
-\Delta u_n + f_n= \lambda_+ \frac{u_n^+}{\left( \eps_n^2 + u_n^2\right)^\frac{2-q}2}-\lambda_- \frac{u_n^-}{\left( \eps_n^2 + u_n^2\right)^\frac{2-q}2}
\]
where
\[
f_n(x)=\frac{2(u_n(x)-\bar u(x))}{1+(u_n(x)-\bar u(x))^4} 
\]
in $B_1$, and hence $u_n$ is smooth in $B_1$, and continuous up to the boundary, and it remains to show that $u_n \weak u$ weakly in $H^1(B_1)$. By minimality, we have that
\begin{equation}\label{18 1 2}
J_n(u_n) \le  J_n(\bar u) = J(\bar u) + \int_{B_1} \left( G(\bar u) - G_n(\bar u) \right) = J(\bar u) + o(1) 
\end{equation}
as $n \to \infty$. This, together with \eqref{18 1 1} and the Poincar\'e inequality, implies that the sequence $\{u_n\}$ is bounded in $H^1(B_1)$, and hence up to a subsequence it converges to a limit $v \in \cK$ weakly in $H^1(B_1)$, strongly in $L^2(B_1)$, and a.e. in $B_1$. In particular,
\[
\int_{B_1} G_n(u_n) \to \int_{B_1} G(v), \quad \int_{B_1} G(u_n) \to \int_{B_1} G(v),
\]
and using the minimality of $\bar u$ 
\begin{equation}\label{18 1 3}
\begin{split}
J(\bar u) &\le J(u_n) = J_n(u_n) + \int_{B_1} \left( G_n(u_n) - G(u_n) \right) - \int_{B_1} \arctan\left( (u_n-\bar u)^2 \right) \\
& = J_n(u_n) + o(1) - \int_{B_1} \arctan\left( (u_n-\bar u)^2 \right).
\end{split}
\end{equation}
Comparing \eqref{18 1 2} and \eqref{18 1 3}, we deduce that
\[
0 \le \int_{B_1} \arctan\left( (u_n-\bar u)^2 \right) \le o(1),
\]
which finally implies that the (point-wise, and hence also weak) limit of $u_n$ is $v= \bar u$.
\end{proof}

\section{Good solutions which are not local minimizers}\label{sec: ex many}

In this section we show that the class of good solutions is much larger than the one of local minimizers. We consider the boundary value problem
\begin{equation}\label{sym bvp}
\begin{cases}
-\Delta u = |u|^{q-2} u & \text{in $B_1$} \\
u = 0 & \text{on $\pa B_1$}, 
\end{cases} \quad \text{with $q \in (0,1)$ and $B_1 \subset \R^2$}.
\end{equation}
This is the homogeneous Dirichlet problem for equation \eqref{eq} in case $\lambda_+=\lambda_-=1$. The natural functional associated with \eqref{sym bvp} is $I: H_0^1(B_1) \to \R$,
\[
I(u) := \int_{B_1} \left( \frac12|\nabla u|^2 - \frac{1}{q} |u|^q \right).
\]
Although $I$ is not differentibale, since $q \in (0,1)$, minimizers of $I$ are good solutions to \eqref{sym bvp}, by Proposition \ref{prop: min are good}. In what follows we show the existence of infinitely many non-minimal good solutions. For $k \in \N$, we denote by $S_k$ the sector
\[
S_k := \left\{ x=\rho( \cos \theta, \sin \theta): \ 0 <\rho < 1, \ 0<\theta<\frac{\pi}{k} \right\},
\] 
and by $R_k$ the rotation of angle $\pi/k$.

\begin{theorem}\label{thm: ex many}
For every $k \in \N$, problem \eqref{sym bvp} has a sign-changing non-minimal good solution with $u_k \ge 0$, $u_k \not \equiv 0$ in $S_k \cup \cdots \cup R_k^{2k-2}(S_k)$, and $u_k \le 0$, $u_k \not \equiv 0$ in $R_k(S_k) \cup \cdots \cup R_k^{2k-1}(S_k)$.  
\end{theorem}

\begin{proof}
Let $k \in \N$ be fixed. We introduce, for any $\eps \in (0,1)$, the differentiable functional $I_\eps: H_0^1(S_k) \to \R$, 
\[
I_\eps(u) := \int_{S_k} \left( \frac12|\nabla u|^2 - \frac{1}{q} \left(\eps^2 + u^2 \right)^\frac{q}2 \right).
\]
It is not difficult to proceed as in \eqref{18 1 1}, showing that
\begin{equation}\label{8 3 1}
I_\eps(u) \ge \frac{1}{2} \int_{S_k} |\nabla u|^2 - C \left( \int_{B_1} |\nabla u|^2 \right)^\frac{q}2 - C \qquad \text{for every $u \in H_0^1(S_k)$}, 
\end{equation}
where $C>0$ is a positive constant independent of $\eps \in (0,1)$. This implies in particular that $I_\eps$ is coercive and bounded from below, and hence for every $\eps \in (0,1)$ there exists a minimizer $u_\eps$, which we can assume to be nonnegative. Since
\[
-\Delta u_\eps = \frac{u_\eps}{\left( \eps^2 + u_\eps^2 \right)^\frac{2-q}2} \quad \text{in $S_k$},
\]
by the strong maximum principle $u_{\eps}>0$ in $S_k$. We obtain after a finite number of odd reflections a solution, still denoted by $u_{\eps}$, in the whole disc $B_1$, with $u_\eps \in H_0^1(B_1)$\footnote{Notice that, since the equation for $u_\eps$ is no more singular, we do not have any problem of regularity in the reflection. Moreover, in principle the reflected function $u_\eps$ solves the equation only in $B_1 \setminus \{0\}$. But being $\{0\}$ a set of zero capacity, this immediately implies that $u_\eps$ is a solution in $B_1$.}. Moreover, by construction $u_{\eps} > 0$ in $S_k \cup \cdots \cup R_k^{2k-2}(S_k)$, and $u_\eps<0$ in $R_k(S_k) \cup \cdots \cup R_k^{2k-1}(S_k)$. Now, let us choose $u \in H_0^1(S_k)$. By minimality, we have for every $\eps \in (0,1)$
\[
I_\eps(u_\eps) \le \inf_{t > 0} I_\eps( t u) \le  \inf_{t > 0} \, \left(\frac{t^2}2 \int_{S_k} |\nabla u|^2 - \frac{t^q}{q} \int_{S_k}|u|^q\right) =: m<0,
\]   
with $m$ independent of $\eps$. Moreover, by \eqref{8 3 1} the family $\{u_\eps\}$ is bounded in $H_0^1(S_k)$, thus in $H_0^1(B_1)$, and hence up to a subsequence $u_\eps \to u_k$ weakly in $H_0^1(B_1)$, strongly in $L^2(B_1)$, and a.e. in $B_1$. The limit $u_k$ inherits the symmetry of $u_\eps$, $u_k \ge 0$ in $S_k$, and moreover $u_k \not \equiv 0$ in $S_k$, since by convergence
\begin{align*}
\int_{S_k} \left( \frac12|\nabla u|^2 - \frac{1}{q} |u|^q \right)  \le \liminf_{\eps \to 0^+} I_\eps(u_\eps) \le m <0.
\end{align*}
Thus, by definition, $u_k$ is a non-trivial good solution to \eqref{sym bvp}, satisfying the sign condition of the thesis, and it remains only to show that $u_k$ is non-minimal. To this end, we observe that if $\bar u$ minimizes $I$ in $H_0^1(B_1)$, so does $|\bar u|$, and by Theorem \ref{thm: reg good} we have then that $|\bar u| \in C^{1,{\frac{q}{2-q}}-}$. If the regular part of the nodal set $\{\bar u=0\}$ is not empty, this is a contradiction. Therefore, it is necessary that $\{\bar u=0\}= \Sigma(\bar u)$ (the nodal set is purely singular), and by Theorem \ref{thm: Hausdorff} it follows that $\{\bar u=0\}$ is the union of finitely many points. But then either $\bar u \ge 0$ in $B_1$, or $\bar u \le 0$ in $B_1$. This proves that minimizers of $I$ in $H_0^1(B_1)$ have constant sign, and, since $u_k$ is sign-changing, completes the proof.
\end{proof}

\begin{remark}
In this construction we strongly used the symmetry of both the equation ($\lambda_+ = \lambda_-=1$) and the domain $B_1$. On the other hand, a multiplicity result similar to holds in a general domain $\Omega \subset \R^N$, making use of Lusternik-Schnirelman theory. We refer the interested reader to Chapter II.5 of \cite{Struwe}, and we only sketch the argument here (see also \cite{NTTVL2} for a related application to a singular perturbation problem). More precisely, let us define
\[
\Gamma_k=\{A\subset H^1_0(\Omega)\setminus\{0\}\;,\; \textrm{$A$ compact; $A=-A$ and } \gamma(A)\geq k\}
\]
where $\gamma$ denotes the Krasnoselskii genus. Let us denote
\[
c_k^\eps=\inf_{A\in\Gamma_k}\max_A I_\eps\;,\quad c_k:=\inf_{A\in\Gamma_k}\max_A I\;.
\]

Note that $c_k\leq c_{k+1}<0$ for every integer $k$; indeed small $k-1$-dimensional spheres of $H^1_0$ belong to $\Gamma_k$ and, by the sublinear character of $G$, $I$ is negative on such spheres.  As $I_\eps< I$ and $I_\eps\to I$ uniformly as $\eps\to 0$, one easily sees that, for every $k\in\N$, $c_k^\eps\to c_k<0$. On the other hand, as soon as $c_k^\eps<0$ it will be critical for $I_\eps$, for $I_\eps$ satisfies the Palais-Smale compactness condition. We then let $\eps\to 0$ to deduce that every level $c_k$ contains at least a good solution. It remains to show that there are infinitely many distinct $c_k$'s. To this aim, note that, whenever $c<0$,  useing again the validity of the Palais-Smale condition for $I_\eps$, we have $\{I\leq c\}\subset\{I_\eps\leq c\}$ and $\gamma(\{I_\eps\leq c\})<+\infty$. This implies that $\lim_{k\to+\infty}c_k\to 0$. It has to be noted that the equation admits only one positive and one negative solutions for $\eps=0$ at level $c_1$, which are the global minimizers of $I$. Hence there are infinitely many sign changing non minimal good solutions.
\end{remark}

\end{document}